\RequirePackage{fix-cm}
\documentclass[smallextended]{svjour3}       % onecolumn (second format)
\smartqed  % flush right qed marks, e.g. at end of proof
\usepackage[misc]{ifsym}
\usepackage{graphicx}

\usepackage{subfigure} % 子图
\usepackage{float}
\usepackage{caption}

\usepackage{graphicx} % 插入eps图片

\usepackage{amssymb,amsmath,mathbbol,amsfonts} % 各种数学符号
\usepackage[dvips,left=3.7cm,right=3.7cm,top=3.7cm,bottom=3.7cm]{geometry} % 页面布局
\usepackage{enumerate} %自定义前导符

\begin{document}
	
	\title{Element-based B-spline basis function spaces: construction and application in isogeometric analysis
	}
	\titlerunning{Element-based B-spline basis function spaces}
	
	
	\author{Peng Yang   \and
		Maodong Pan \and 
		Falai Chen \and
		Zhimin Zhang }
	
	
	\institute{Peng Yang \at
		School of Mathematical Sciences, University of Electronic Science and Technology of China, Chengdu, Sichuan 611731, China \\
		\email{pyang@uestc.edu.cn}           %  \\
		\and
		Maodong Pan $ (\textrm{\Letter}) $ \at 
		School of Mathematics, Nanjing University of Aeronautics and Astronautics, Nanjing, China\\
		\email{mdpan@mail.ustc.edu.cn}
		\and
		Falai Chen \at
		School of Mathematics, University of Science and Technology of China, Hefei, Ahui 230026, China \\
		\email{chenfl@ustc.edu.cn} 
		\and
		Zhimin Zhang \at 
		Department of Mathematics, Wayne State University, Detroit, Michigan, 48202, USA \\
		\email{ag7761@wayne.edu}    
	}
	
	\date{Received: date / Accepted: date}

	\maketitle
	
	\begin{abstract}
		This paper develops a unified theoretical framework for constructing B-spline basis function spaces with structural equivalence to finite element spaces. The theory rigorously establishes that these bases emerge as explicit linear combinations of B-spline element bases. For any prescribed smoothness requirements, this element-wise formulation enables the Hermite interpolation at nodes, which directly utilizes function values and derivatives without solving global linear systems.
		By focusing on explicit interpolation properties, element-wise analysis establishes optimal approximation errors, even when the space smoothness attains its theoretical maximum for the space degree. In isogeometric analysis (IgA), the construction naturally decomposes geometric mappings into element-level representations, allowing efficient computations across elements regardless of node distribution. Notably, the same Hermite interpolation framework simultaneously handles domain parameterization and IgA solutions, allowing direct imposition of boundary conditions through function and derivative matching. Numerical tests demonstrate optimal convergence rates and superconvergence properties in 2D IgA under uniform knot configurations, and improved computational efficiency in 3D IgA with non-uniform knot distributions.
		
	\end{abstract}
	
	\section{Introduction}
	\label{intro}
	
	The B-spline basis function space has been extensively applied across diverse fields including data analysis, visualization, geometric modeling, and numerical simulations. As a high-smoothness space, it serves as a frequent choice for finite element solution spaces in discretizing high-order partial differential equations or enforcing high-continuity conditions \cite{Hollig2003,Hollig2001,Hu2015,Schaeuble2017,Chung1979}. In Computer-Aided Design (CAD), B-splines attain significant geometric modeling capabilities through their extension to non-uniform rational B-splines (NURBS) \cite{piegl2012nurbs}. The concept of isogeometric analysis (IgA), introduced in \cite{Hughes2009,Hughes2005IGA}, integrates CAD and numerical computation by using NURBS-defined geometric domains to approximate unknown solutions of partial differential equations. This approach shares a conceptual alignment with finite element methods (FEM), cf.  \cite{hughes2000finite}. The similarity particularly manifests in IgA's basis functions and the isoparametric elements of FEM through their common parametric-to-physical mapping framework. However, their parametric mapping mechanisms differ fundamentally: FEM-based isoparametric elements rely on piecewise local mappings defined per element, whereas IgA directly inherits global geometric parameterization from CAD systems. 
	The inherent geometric exactness of IgA has spurred extensive investigations \cite{Bazilevs2006,Guo2015,Sange2019,Sande2020,Beirao2014,Vuong2011}.
	
	For general knot vector configurations, the recursive definition of B-spline basis functions leads to computational inefficiency in numerical implementations.  In order to enhance practicality, \cite{Borden2011} proposed Bézier extraction, representing B-spline basis functions as element-wise Bernstein polynomials. This enables smoothly integrating isogeometric computations into finite element assembly procedures. For uniform knot vectors, explicit representations of B-spline basis functions permit efficient numerical algorithms via precomputed lookup tables for basis function inner products \cite{Antolin2015,Calabro2017,Pan2020}, with \cite{Pan2020} further approximating variable coefficients using B-spline quasi-interpolation. Nevertheless, unlike finite element methods that inherently utilize element-wise bases (e.g., $ C^0 $ Lagrange or $ C^p $ Hermite elements) with explicit interpolation properties to facilitate localized computations, B-spline-based approaches face challenges in achieving unified local operations. 
	
	The macro-element concept enables spline spaces to adopt element-wise basis definitions, as demonstrated for triangular/tetrahedral meshes, e.g., \cite{Alfeld2009,Cohen2013,Lai2007,Speleers2015,Lyche2022}. However, existing constructions often require polynomial degrees exceeding the theoretical minimum ($ p+1 $ for $ C^p $ continuity), with optimal-degree polynomial spaces remaining elusive until recent advances in $ C^2 $ cubic splines on triangular meshes \cite{Lyche2022FCM}. In one dimension or tensor-product settings, macro-elements combine intervals or rectangles into unified $``$elements" with endpoint Hermite interpolants. This process achieves optimal approximation estimates \cite{Hu2015} and demonstrates efficacy in high-frequency dynamics \cite{Zak2021}. Notably, such macro-element-defined bases generally differ from global B-spline bases. In summary,  incorporating the local advantages of macro-elements into B-spline basis function spaces would enable locally enhanced analytical properties (e.g., refined approximation power) and locally optimized computational procedures, while preserving global continuity requirements.
	
	This work proposes a novel B-spline space construction via macro-elements, termed B-spline elements, under arbitrary knot sequences (including repeated knots). Specifically, we first define local basis functions for B-spline elements, where each basis is restricted to its corresponding element domain. Then, the global B-spline basis is reconstructed by linearly combining these local bases across adjacent B-spline elements, governed by a recursive relationship (see Theorem~\ref{theorem, generate Bspline}). This process inherently unifies element-wise definitions with the classical B-spline framework. Moreover, we establish Hermite interpolation schemes for arbitrary smoothness constraints and prove the optimal error estimates by leveraging element analysis. This resolves a critical limitation in \cite{Beirao2011}, which, despite adopting similar element-wise interpolation strategies, requires polynomial degrees exceeding the theoretical minimum (e.g., $ p+1 $ for  $ C^{p} $ spaces) to achieve error bounds. Furthermore, the element-wise framework  enhances computations by decomposing CAD-generated global geometric mappings into local element mappings, aligning with finite element assembly workflows.  Variable coefficients in inner product integrals are locally approximated via Hermite interpolation, eliminating reliance on global quasi-interpolation \cite{Pan2020}. Combined with precomputed lookup tables for basis integrals, this approach removes knot uniformity constraints in prior table-driven methods \cite{Antolin2015,Calabro2017,Pan2020}. In addition, the explicit Hermite interpolation properties are applied to both the parameterised geometric mappings and the isogeometric solution spaces, allowing direct enforcement of high-order boundary constraints (e.g., displacement, rotation, and bending moment conditions in biharmonic problems). 
%	This eliminates  approximation errors inherent in conventional weak enforcement methods.
	
	The rest of this paper is structured as follows. Section~\ref{sec1} reformulates B-spline basis function spaces via B-spline element bases. Section~\ref{sec2} subsequently addresses their Hermite interpolation properties and proves the optimal approximation errors. Section~\ref{sec3} implements the framework in isogeometric analysis with localized computations. Section~\ref{sec4} presents numerical validations, followed by conclusions and future directions in Section~\ref{sec5}.
	\section{Element-based B-spline basis function spaces}
	\label{sec1}
	%% Labels are used to cross-reference an item using \ref command.
	
	In this section, the B-spline basis function spaces will be constructed in a novel manner, which exhibits significant parallels  to the formulation of finite element spaces.
	The resulting basis functions emerge as linear combinations of the basis functions of B-spline elements. The discussion in this section is focused on one-dimensional spaces; however, for higher dimensional spaces, consideration is given directly to their tensor-product forms. The description of higher dimensional tensor-product  spaces is omitted here, as it is a natural concept and has been comprehensively covered in the majority of literature on this subject.
	
	A concise overview of B-spline basis function spaces is presented in  Subsection~\ref{subsec1.1}. Subsection~\ref{subsec1.2} then delineates a novel approach to their construction via B-spline elements, while Subsection~\ref{subsec1.3} provides an in-depth analysis of the explicit forms of B-spline element basis functions.
	%% Use \subsection commands to start a subsection.
	\subsection{Overview of B-spline basis function spaces}
	\label{subsec1.1}
	We first introduce the knot vector
	\begin{eqnarray*}
		\Xi: = \{ \underbrace{\xi_{1},...,\xi_{1}}_{\mbox{$ r_1 $ times}},\underbrace{\xi_{2},...,\xi_{2}}_{\mbox{$ r_2 $ times}},...,\underbrace{\xi_{m},...,\xi_{m}}_{\mbox{$ r_m $ times}} \},
	\end{eqnarray*}
	where $ \xi_{1}< \xi_{2}<...<\xi_{m}$ are  distinctive  knots. Let $ r\xi $ denote a repeating $ \xi $,  occurring  $ r $ times.  Then the  knot vector $ \Xi $ simply reads as
	\begin{eqnarray*}
		\Xi: = \{ r_1\xi_{1},r_2\xi_{2},...,r_m\xi_{m} \}.
	\end{eqnarray*}
	Given a positive integer $ p  $, we consider  $ r_1=r_m=p+1 $ and $  r_2,...,r_{m-1}\leq p  $, in which case $ \Xi $ is called an \textit{open} knot vector.  
	
	Within the context of $ \Xi $, the B-spline basis function space of degree $ p $ can be constructed. In fact, upon noting  the knots of $ \Xi $ as a sequence of nondecreasing knots $ \{t_i,\,i=1,2,...,\sum_{a=1}^{m}r_a\} $ and letting $ T_{i}:=\{t_i,t_{i+1},...,t_{i+p+1}\} $ be the $ i $th knot vector consisting of $ p+2 $ consecutive knots selected from  $ \Xi $, the $ i $th B-spline basis function $ B_i(x) $ is obtained by $ T_{i}$  using  the following recurrence relation, e.g., in \cite{Boor1978,Hughes2005IGA},
	\begin{eqnarray}\label{recurrence relation,1}
		N_{j,0}(x)= \left\{
		\begin{aligned}
			&1,& t_{i+j-1}\leq x<t_{i+j},\\
			&0,& \mbox{otherwise}
		\end{aligned}
		\right. \qquad \mbox{for $ 1\leq j\leq p+1 $},
	\end{eqnarray}
	and 
	\begin{equation}\label{recurrence relation,2}
		\begin{split}
			N_{j,m}(x)= &\frac{x-t_{i+j-1}}{t_{i+j+m-1}-t_{i+j-1}}N_{j,m-1}(x)+\frac{t_{i+j+m}-x}{t_{i+j+m}-t_{i+j}}N_{j+1,m-1}(x), \\
			&\mbox{for $ 1\leq j\leq p+1-m,  1\leq m\leq p$},
		\end{split}
	\end{equation}
	then letting $ B_{i}(x):=	N_{1,p}(x) $. 
	%		In instances where repeated knots are present within $ T_{i}$ for \eqref{recurrence relation,2}, a definition shall be provided as follows
	%		\begin{eqnarray*}
		%			\frac{\cdot}{t_i-t_j}=0,  \quad \mbox{if $ t_i=t_j $}.
		%		\end{eqnarray*}
	
	Denote the set of B-spline basis functions  by
	\begin{eqnarray*}
		\mathcal{B}_{p}(\Xi): = \{B_{i}(x),  i=1,2,...,N_{\Xi}\}.
	\end{eqnarray*}
	where $ N_{\Xi}= \sum_{a=1}^{m}r_a-(p+1) $.
	For some important properties of $ \mathcal{B}_{p}(\Xi) $, we mention
	\begin{itemize}
		\item $ B_{i}(x) $ has the support $ \langle T_i \rangle $, which denotes the convex hull of $ T_i $.
		\item $ B_{i}(x) $ is a nonnegative piecewise polynomial with respect to the  distinctive  knots of $ T_i $.
		\item $ B_{i}(x) $ has $ C^{p-m} $ smoothness across a knot $ \xi $ of $ T_i $, where $ m$ is the repetition times of $ \xi $ in $ T_i $.
		\item On each interval $ (\xi_j,\xi_{j+1}) $, there are exactly $ p+1 $ B-spline basis functions having nonzero values. Moreover, the B-spline basis  forms a polynomial space of total degree $ p $ on $ \langle \Xi \rangle $.
		\item  The B-spline basis functions defined by $ \Xi $ form a partition of unity, i.e., for $ x \in \langle \Xi \rangle$,
		$ \sum_{i=1}^{N_{\Xi}}B_{i}(x)=1$.
	\end{itemize}
	The first three properties show the individual properties of B-spline basis functions, while the last two properties indicate their global properties on  $ \Xi $. In the general case where $ \Xi $ is not an open knot vector, a sequence of B-spline functions can nevertheless be obtained by means of recurrence relation \eqref{recurrence relation,1}-\eqref{recurrence relation,2}. However, these B-splines fail to satisfy the last two global properties mentioned above. The B-spline basis function space of degree $ p $ defined on the open knot vector $ \Xi $ is denoted by
	\begin{equation*}
		\mathcal{S}_{p}^{\boldsymbol{k}}(\Xi): = \text{span}\{\mathcal{B}_{p}(\Xi)\},
	\end{equation*}
	where the  vector $ \boldsymbol{k}:=\{k_1,k_2,...,k_m\}  $ captures the continuity of derivatives at knots $ \xi_1,...,\xi_m $ to the minimum degree. That is, $ k_i = p-r_i $ for $ i = 1,...,m $.
	
	In the end of this subsection, the knot insertion is addressed.  Clearly, it is sufficient to consider the insertion of one knot at a time since any (finite) refinement can be achieved by adding single knots iteratively \cite{Boor1978}.
	\begin{lemma}[Knot insertion]\label{lemma,knot insertion}
		Let $ \Xi:=\{t_i\}_{i=1}^{p+n+1} $ be a knot vector with nondecreasing sequence of knots, where $ p $ is the polynomial degree of the B-spline sequence defined on $ \Xi $, and $ n $ is the number of B-splines. Consider a knot  $ t^{*} $ which   satisfies $ t_{m-1}\leq t^{*}< t_m  $ for a certain index $  m $. Through the assignment
		\begin{eqnarray*}
			t_i^* = \left\{
			\begin{aligned}
				&t_i,& i<m,\\
				&t^*,& i=m,\\
				&t_{i-1},& i>m,\\
			\end{aligned}
			\right. \quad \forall 1\leq i\leq p+n+2,
		\end{eqnarray*}
		one can define the refined knot vector $ \Xi^*:=\{t^*_i\}_{i=1}^{p+n+2} $. Let $ \{a_i\}_{i=1}^{n} $ be any sequence of coefficients in $ \mathbb{R} $, and define
		\begin{eqnarray*}
			a^*_i:=(1-w_i)a_{i-1}+w_ia_i, \quad \mbox{where}\quad w_i = \left\{
			\begin{aligned}
				&0,&& \mbox{if \,\,$ t_i\geq t^* $},\\
				&\frac{t^*-t_i}{t_{i+p}-t_i},&& \mbox{if\,\, $ t_i < t^*< t_{i+p} $},\\
				&1,&& \mbox{if \,\,$ t_{i+p}\leq t^* $},\\
			\end{aligned}
			\right. 
		\end{eqnarray*}
		for $ 1\leq i \leq n+1 $ to obtain another sequence of coefficients $ \{a^*_i\}_{i=1}^{n+1} $  in $ \mathbb{R} $. Denote the B-spline sequence of degree $ p $ by $ \{B_i(x)\}_{i=1}^{n} $ for  $ \Xi $ and  $ \{B^*_i(x)\}_{i=1}^{n+1} $ for  $ \Xi^* $. Then, the following identity holds true.
		\begin{eqnarray}\label{eq:identity}
			\sum_{i=1}^{n}a_iB_i(x)=\sum_{i=1}^{n+1}a^*_iB^*_i(x).
		\end{eqnarray}
	\end{lemma}
	
	For an open knot vector $ \Xi $, the knot insertion process does not result in any changes to the basis functions whose support excludes the newly inserted knot. In particular, when the inserted knot increases the multiplicity number of a knot in $ \Xi $, the B-spline basis functions which are modified at the right end of the identity \eqref{eq:identity} are of a lower order of smoothness.  As demonstrated in literature \cite{Borden2011}, inserting (repeating) knots as many times as is necessary can ensure that the modified B-spline basis functions become Bernstein basis functions. In this manner, each B-spline basis function can be represented by a linear combination of Bernstein basis functions, which possess local explicit forms. This process of decomposition is referred to as \textit{B\'{e}zier extraction}. In numerical computations, the utilisation of B\'{e}zier extraction technique to obtain B-spline basis functions can be more easily incorporated into existing finite element codes than the traditional recurrence relation \eqref{recurrence relation,1}-\eqref{recurrence relation,2}.
	
	%% Use \subsubsection, \paragraph, \subparagraph commands to 
	%% start 3rd, 4th and 5th level sections.
	%% Refer following link for more details.
	%% https://en.wikibooks.org/wiki/LaTeX/Document_Structure#Sectioning_commands
	
	\subsection{Element-based construction of B-spline spaces}\label{subsec1.2}
	We consider the following open knot vector:
	\begin{equation}\label{eq:knot vector}
		\Xi^{*} =\{ r_1\xi_{1},\xi_{1,1},...,\xi_{1,p-1},r_2\xi_{2},...,r_{m-1}\xi_{m-1},\xi_{m-1,1},...,\xi_{m-1,p-1},r_m\xi_{m} \},
	\end{equation}
	where $ r_1=r_m=p+1 $ and $  r_2,...,r_{m-1}\leq p  $.
	Here, for $ j=1,...,m-1 $, the knots $ \xi_{j,1},...,\xi_{j,p-1} $ equalize the interval $ (\xi_{j},\xi_{j+1}) $ into $ p $ subintervals.  Indeed, it is possible to convert any knot vector $ \Xi $ to the above form by inserting appropriate knots. It can be deduced from Lemma~\ref{lemma,knot insertion} that, for any function in the B-spline basis function space $ \mathcal{S}_{p}^{\boldsymbol{k}}(\Xi) $, there is a  corresponding function in $ \mathcal{S}_{p}^{\boldsymbol{k}^{*}}(\Xi^{*}) $ such that they have same values everywhere. Without causing confusion, we will subsequently refer to $ \Xi^* $ by $ \Xi $. 
	
	In the following, we proceed with the construction of $ \mathcal{S}_{p}^{\boldsymbol{k}}(\Xi) $ through the B-spline elements, exhibiting significant parallels  to the formulation of the finite element spaces. The basis functions of $ \mathcal{S}_{p}^{\boldsymbol{k}}(\Xi) $ will be represented as linear combinations of B-spline element basis functions. This process of linear combinations is analogous to B\'{e}zier extraction  in that they are both special cases of knot insertion.
	
	For $ j=1,...,m-1 $, let $ K_j:=(\xi_{j},\xi_{j+1}) $ and $ h_j:= \xi_{j+1}-\xi_{j} $. Then $ \mathcal{T}_{h}:=\left\lbrace K_{j}\right\rbrace_{j=1}^{m-1} $ is a partition of $\langle \Xi \rangle $ with $h=\max{h_{j}}$. Denote the open knot vectors 
	\begin{eqnarray*}
		\Xi_{K_j}:=\{(p+1)\xi_{j,0},\xi_{j,1},...,\xi_{j,p-1},(p+1)\xi_{j,p} \},\quad  j=1,...,m-1,
	\end{eqnarray*}
	where $ \xi_{j,0}: = \xi_j $ and $ \xi_{j,p}: = \xi_{j+1} $.
	Let $ T_{K_j,i} $ be the $ i $th knot vector consisting of $ p+2 $ consecutive knots selected from  $ \Xi_{K_j} $, i.e.,
	\begin{eqnarray}\label{def:element knot basis 1}
		T_{K_j,i}:=\left\lbrace(p+2-i)\xi_{j,0},\xi_{j,1},..,\xi_{j,i}\right\rbrace, \quad  i=1,...,p,
	\end{eqnarray}
	and 
	\begin{eqnarray}\label{def:element knot basis 2}
		T_{K_j,i}:=\left\lbrace\xi_{j,i-p-1},..,\xi_{j,p-1},(i-p+1)\xi_{j,p}\right\rbrace, \quad i=p+1,...,2p.
	\end{eqnarray}
	It is then readily to determine the corresponding local B-spline basis functions 
	\begin{equation}\label{eq:element B-spline}
		\mathcal{B}_{p}(\Xi_{K_j}):=\{ B_{K_j,i},i=1,...,2p\},\quad  j=1,...,m-1,
	\end{equation}
	where $  B_{K_j,i} $ is obtained by $ T_{K_j,i}$  using  the  recurrence relation \eqref{recurrence relation,1}-\eqref{recurrence relation,2}, and the spaces
	\begin{equation*}
		\mathcal{S}_{p}^{\boldsymbol{k}}(\Xi_{K_j}):= \text{span}\{\mathcal{B}_{p}(\Xi_{K_j})\},\quad  j=1,...,m-1.
	\end{equation*}
	We call the triple $(K_j,\mathcal{S}_{p}^{\boldsymbol{k}}(\Xi_{K_j}),\mathcal{B}_{p}(\Xi_{K_j}))  $  a B-spline element of degree $ p $. The smoothness property of the  B-spline element basis functions $ \mathcal{B}_{p}(\Xi_{K_j}) $ at the distinctive knots of $ \Xi_{K_j}$ can be summarised as follows.
	\begin{itemize}
		\item For $ 1\leq i\leq p $, $ B_{K_j,i}(x) $  has $ C^{i-2} $ smoothness across $ \xi_{j,0} $ and $ C^{p-1} $ smoothness across other knots of $ \Xi_{K_j}$;
		\item For $ p+1\leq i\leq 2p $, $ B_{K_j,i}(x)  $ has $ C^{2p-1-i} $ smoothness across $ \xi_{j,p} $ and $ C^{p-1} $ smoothness across other knots of $ \Xi_{K_j}$.
	\end{itemize}
	For the cases $ p=1,2,3,4 $, the B-spline basis functions of a B-spline element  are illustrated in Fig.~\ref{figse}.
	\begin{figure}[htbp]%% placement specifier
		%% Use \includegraphics command to insert graphic files. Place graphics files in 
		%% working directory.
		\centering%% For centre alignment of image.
		\subfigure[$ p=1 $]{
			\begin{minipage}{.45\textwidth}
				\centering
				\includegraphics[width=140pt]{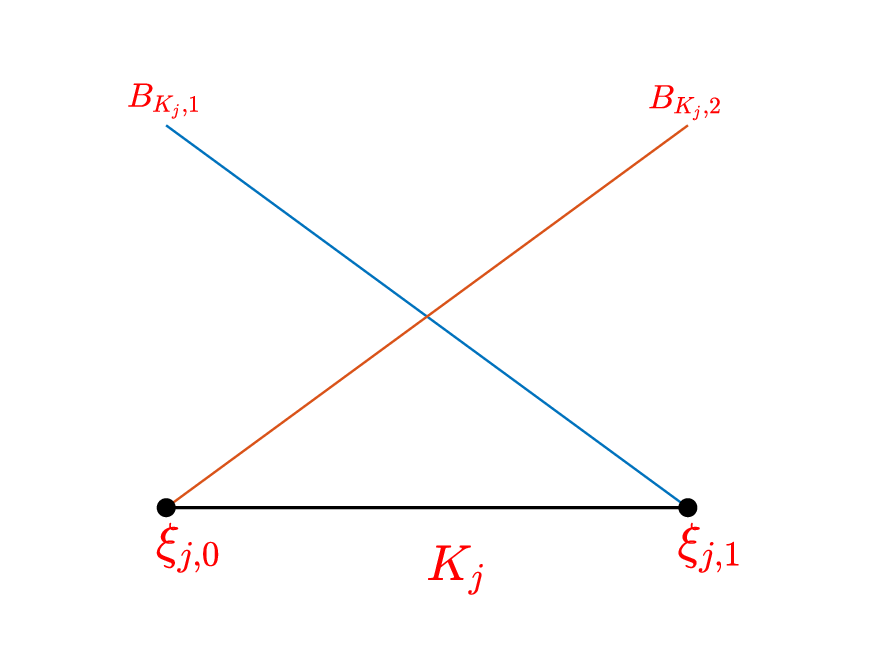}
				\label{figse1}
			\end{minipage}
		}
		\subfigure[$ p=2 $]{
			\begin{minipage}{.45\textwidth}
				\centering
				\includegraphics[width=140pt]{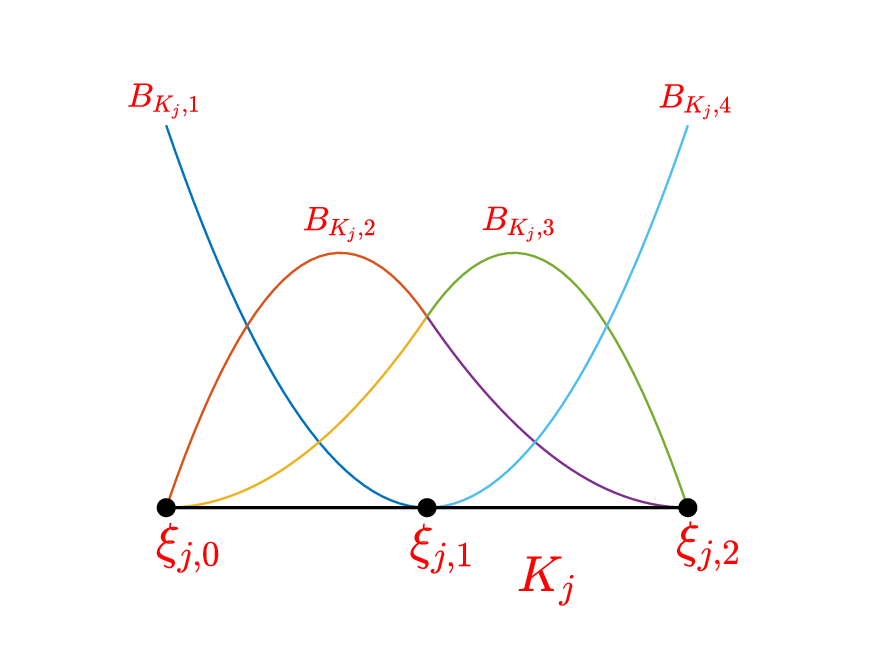}
				\label{figse2}
			\end{minipage}
		}\\
		\subfigure[$ p=3$]{
			\begin{minipage}{.45\textwidth}
				\centering
				\includegraphics[width=140pt]{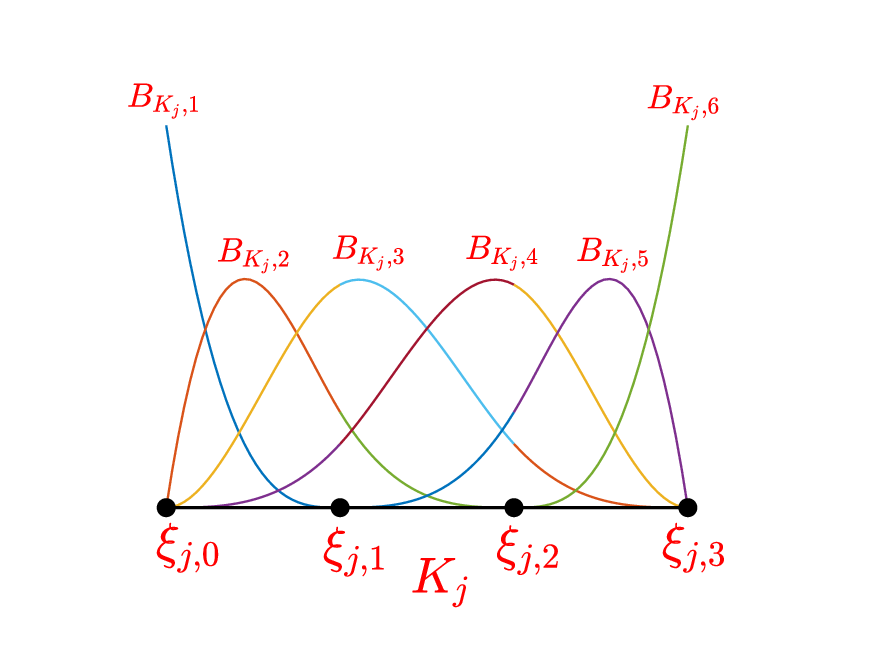}
				\label{figse3}
			\end{minipage}
		}
		\subfigure[$ p=4 $]{
			\begin{minipage}{.45\textwidth}
				\centering
				\includegraphics[width=140pt]{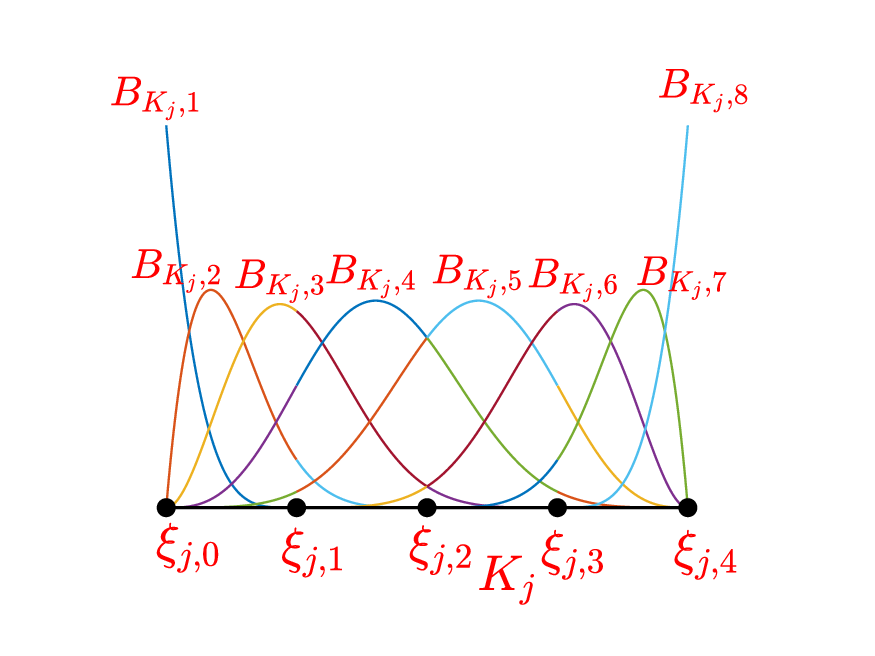}
				\label{figse4}
			\end{minipage}
		}
		%% Use \caption command for figure caption and label.
		\caption{Illustration of basis functions for the B-spline element $(K_j,\mathcal{S}_{p}^{\boldsymbol{k}}(\Xi_{K_j}),\mathcal{B}_{p}(\Xi_{K_j}))  $.}\label{figse}
		%% https://en.wikibooks.org/wiki/LaTeX/Importing_Graphics#Importing_external_graphics
	\end{figure} 
	
	With respect to $ \mathcal{T}_{h} $, the objective here is to derive $ \mathcal{B}_{p}(\Xi) $ using $ \mathcal{B}_{p}(\Xi_{K_j}) $. To this end, we consider the combinations of  B-spline element basis functions as the inverse process of knot insertion in Lemma~\ref{lemma,knot insertion}, aiming to derive B-splines with higher smoothness. We have the following Theorem~\ref{theorem, generate Bspline}.
	\begin{theorem}\label{theorem, generate Bspline}
		Let $ K_{j-1} $ and $ K_{j} $ be two neighboring elements, and consider the  B-spline element basis functions $ \{B_{K_{j-1},i}(x)\} _{i=p+1}^{2p}  $ and $ \{B_{K_{j},i}(x)\} _{i=1}^{p}  $  of associated B-spline elements. By setting
		\begin{eqnarray*}
			B_{m,0,j}(x) = B_{K_{j-1},2p-m}(x), \quad B_{m,m+1,j}(x) = B_{K_{j},m+1}(x) \quad \forall 0\leq m \leq p-1,
		\end{eqnarray*}
		and using the recurrence relation (see Fig.~\ref{fig: recurrence relation})
		\begin{equation}\label{eq,recurrence relation smoothness}
			\begin{split}
				B_{m,n,j}(x) =H_{m,n-1,j}&B_{m-1,n-1,j}(x)+(1-H_{m,n,j})B_{m-1,n,j}(x)\,\,\,\\ &\forall 1\leq n \leq m, 1\leq m \leq p,
			\end{split}
		\end{equation}
		where 
		\begin{eqnarray*}
			H_{m,n,j} = \frac{(m-n)h_{j-1}}{(m-n)h_{j-1}+nh_{j}},
		\end{eqnarray*}
		we have the resulting functions
		\begin{eqnarray*}
			B_{m,n,j}(x)\in C^{m-1}(K_{j-1}\cup K_{j}) \quad \forall 1\leq n \leq m, \,1\leq m \leq p,
		\end{eqnarray*}
		which are B-splines.
	\end{theorem}
	\begin{figure}[htbp]%% placement specifier
		%% Use \includegraphics command to insert graphic files. Place graphics files in 
		%% working directory.
		\centering
		\includegraphics[width=240pt]{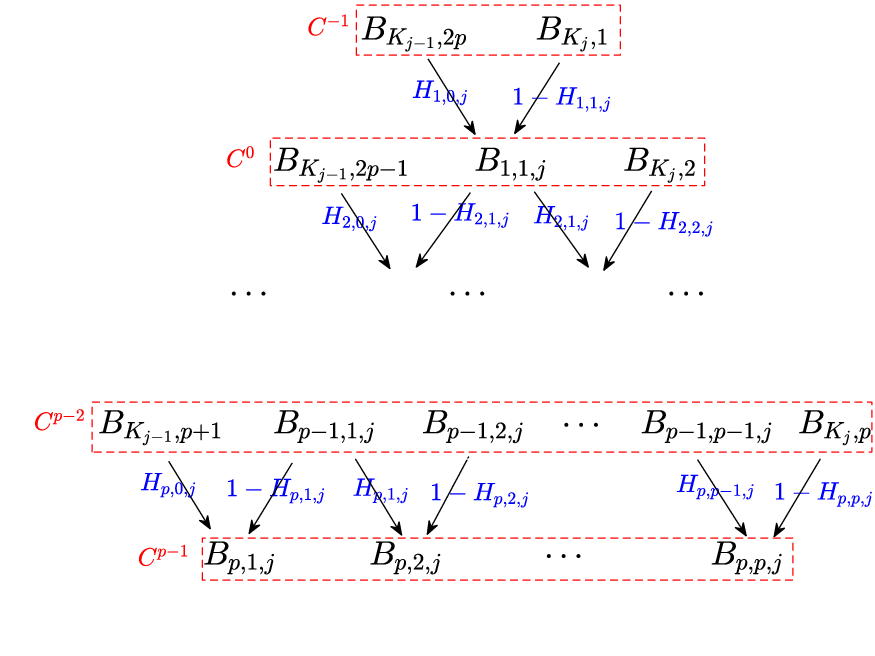}
		\caption{Recurrence relation via B-spline elements}
		\label{fig: recurrence relation}
	\end{figure}
	\begin{proof}
		For a positive integer $ r\leq p $, consider the knot vector $ \Xi_{r}:=\{t_i\}_{i=1}^{2p+r} $ which is dependent on $ r $, with
		\begin{eqnarray*}
			(\xi_{j-1}=)t_1<...<t_{p}<t_{p+1}=...=t_{p+r}(=\xi_{j})<t_{p+r+1}<...<t_{2p+r}(=\xi_{j+1}),
		\end{eqnarray*}
		and the refined knot vector $ \Xi^{*}_{r}:=\{t^*_i\}_{i=1}^{2p+r+1} $, dependent on $ r $, with
		\begin{eqnarray*}
			t^{*}_1<...<t^{*}_{p}<t^{*}_{p+1}=...=t^{*}_{p+r+1}<t^{*}_{p+r+2}<...<t^{*}_{2p+r+1},
		\end{eqnarray*}
		where $ t^{*}_{p+r+1} $ is the  inserted knot. Here, $ \{t_i\}_{i=1}^{p+1} $ (or $ \{t^*_i\}_{i=1}^{p+1} $) and $ \{t_i\}_{i=p+r}^{2p+r} $ (or $ \{t^*_i\}_{i=p+r+1}^{2p+r+1} $) are equidistant knots in $ K_{j-1} $ and $ K_{j} $, i.e.,
		\begin{equation*}
			t_{i+1}- t_{i} = \frac{h_{j-1}}{p}, \quad  t_{p+r+i}- t_{p+r+i-1} = \frac{h_{j}}{p},\quad\,\,\, i = 1,2,...,p.
		\end{equation*}
		Let $ \{B_i(x)\}_{i=1}^{p+r-1} $ and $ \{B^*_i(x)\}_{i=1}^{p+r} $ be the B-spline sequences for  $ \Xi_{r} $ and  $ \Xi_{r}^* $, respectively. It follows from Lemma~\ref{lemma,knot insertion} that
		\begin{eqnarray}\label{eq: equi S and S*}
			\sum_{i=1}^{p+r-1}a_iB_i=\sum_{i=1}^{r}a_iB^*_i+\sum_{i=r+1}^{p}\big((1-w_i)a_{i-1}+w_ia_i\big)B^*_i+\sum_{i=p+1}^{p+r}a_{i-1}B^*_i,
		\end{eqnarray}
		where, for $ i=r+1,r+2,...,p $,
		\begin{eqnarray*}
			w_i = \frac{t_{p+1}-t_i}{t_{p+i}-t_i}=\frac{t_{p+1}-t_i}{(t_{p+i}-t_{p+r})+(t_{p+1}-t_i)}=\frac{(p+1-i)h_{j-1}}{(i-r)h_{j}+(p+1-i)h_{j-1}}.
		\end{eqnarray*}
		Note that, for $ 1\leq i\leq r-1 $, the B-splines $ B_i $ and $ B_i^* $ are defined by the same knot vectors $ \left\lbrace t_{i},t_{i+1},...,t_{p+i+1} \right\rbrace   $, and for $ p+1\leq i\leq p+r-1 $, the  B-splines $ B_{i} $ and $ B_{i+1}^* $ are defined by the same knot vectors $ \left\lbrace t_{i},t_{i+1},...,t_{p+i+1} \right\rbrace   $. Thus, \eqref{eq: equi S and S*} is equivalent to
		\begin{eqnarray*}
			\sum_{i=r}^{p}a_iB_i=a_rB^*_r+\sum_{i=r+1}^{p}\big((1-w_i)a_{i-1}+w_ia_i\big)B^*_i+a_{p}B^*_{p+1}.
		\end{eqnarray*}
		Given that the coefficients $ a_i $ are arbitrary, a comparison of the two sides of the above equation yields 
		\begin{equation}\label{eq,generation process B}
			B_i = w_{i}B^*_i+(1-w_{i+1})B^*_{i+1}\quad \forall r\leq i\leq p.
		\end{equation}
		This formula can be regarded as a combination of $ p+2-r $ functions $ B_i^*(x) $ $ (r\leq i\leq p+1) $, each having $ C^{p-1-r} $ smoothness across $ t_{p+r+1}^* $. These functions collectively generate $ p+1-r $ functions $ B_i(x) $ $ (r\leq i\leq p) $, each having $ C^{p-r} $ smoothness across $ t_{p+r+1}^* $. In consideration of the combination process from $ r=p $ to $ r=1 $, the proof of Theorem~\ref{theorem, generate Bspline} can be derived.
	\end{proof}
	\begin{remark}
		Note that the recurrence relation \eqref{eq,recurrence relation smoothness}  involves only the second half of B-spline element basis functions  $\{B_{K_{j-1},i}(x)\}_{i=p+1}^{2p}$ on $ K_{j-1}$ and the first half of B-spline element basis functions  $\{B_{K_{j},i}(x)\}_{i=1}^{p}$ on $ K_{j}$, which do not have optimal $ C^{p-1} $ smoothness across the common endpoint of $ K_{j-1} $ and $ K_j $.
	\end{remark}
	%		\begin{figure}[htbp]
		%			\centering
		%			\includegraphics[width=240pt]{figures/recurFig.eps}
		%			\caption{Illustration of the recurrence relation \eqref{eq,recurrence relation smoothness}.}
		%			\label{fig: recurrence relation}
		%		\end{figure}
	In the proof of Theorem~\ref{theorem, generate Bspline}, the B-spline sequence defined on $ \Xi_{r} $  shall hereafter be denoted  by $ \mathcal{B}_{\xi_j,r} $, i.e., (see Fig.~\ref{fig: recurrence relation})
	\begin{equation*}
		\mathcal{B}_{\xi_j,r}:=	\{B_{K_{j-1},p+1},...,B_{K_{j-1},p+r-1},B_{p-r+1,1,j},...,B_{p-r+1,p-r+1,j},B_{K_{j},p-r+2},...,B_{K_{j},p} \}.
	\end{equation*}
	Clearly,  the B-splines of  $ \mathcal{B}_{\xi_j,r} $ exhibit a range of at least $ C^{p-r} $ and at most $ C^{p-1} $ smoothness across the common knot $ \xi_j $ of elements $ K_{j-1} $ and $ K_{j} $. We introduce two virtual elements $K_{0}:=(\xi_1,\xi_1)$ and  $K_{m}: =(\xi_{m},\xi_{m})$ with the lengths $h_{0}=h_{m}=0$. This ensures that every knot can be described as the common knot of a pair of elements. It is now possible to represent the B-spline basis functions of $ \mathcal{S}_p^{\boldsymbol{k}}(\Xi) $ in the following manner
	\begin{equation}\label{eq:element-based basis}
		\mathcal{B}_p(\Xi) = \{\mathcal{B}_{\xi_1,r_1},\mathcal{B}_{\xi_2,r_2},...,\mathcal{B}_{\xi_m,r_m}  \},
	\end{equation}
	where each basis function is a linear combination of B-spline element basis functions, with coefficients being related only to the lengths of elements.
	It is easy to find that the dimension of $ \mathcal{S}_p^{\boldsymbol{k}}(\Xi) $ is
	\begin{equation*}
		\dim(\mathcal{S}_p^{\boldsymbol{k}}(\Xi)) = (p+r_1-1) +...+(p+r_m-1) = m(p-1)+\sum_{a=1}^{m}r_a.
	\end{equation*}
	In the case of $ r_2=...=r_{m-1}=1 $, we have $ \mathcal{S}_p^{\boldsymbol{k}}(\Xi)\subset C^{p-1} (\langle \Xi \rangle) $.
	We also denote  by $ \mathcal{S}_p^{p-1}(\mathcal{T}_{h}) $ such a $ C^{p-1} $ finite element space and $ \mathcal{B}_p(\mathcal{T}_{h}) $ the set of B-spline basis functions.
	
	%To obtain B-splines of degree $ p $, it is important to distinguish that the recurrence relation \eqref{recurrence relation,1}-\eqref{recurrence relation,2} starts with piecewise constants, while the established recurrence relation \eqref{eq,recurrence relation smoothness} starts with $ C^{-1} $ piecewise polynomials of degree $ p $.
	\subsection{Explicit representations of B-spline element bases}\label{subsec1.3}
	The objective here is to derive the explicit representations of B-spline element basis functions $ \mathcal{B}_{p}(\Xi_{K_j}) $ given by \eqref{eq:element B-spline}. Let
	\begin{equation*}
		\lambda_1(x) = \frac{\xi_{j,p} -x}{h_j},\quad\quad \lambda_2(x) = \frac{x-\xi_{j,0} }{h_j} \quad \forall x\in K_j.
	\end{equation*}
	Then the knots $ \xi_{j,0},\xi_{j,1},...,\xi_{j,p} $  can be represented as
	\begin{equation*}
		\xi_{j,s}:\,s\lambda_1-(p-s)\lambda_2=0,\quad 0\leq s\leq p.
	\end{equation*}
	%		It is observed that the piecewise polynomial $ B_{K_j,i}(\cdot) $ has  $C^{p-1} $ smoothness across $ \xi_{j,s} $, $ 0\leq s\leq p $, if and only if
	%		\begin{equation*}
		%		B_{K_j,i}(\lambda_1,\lambda_2)|_{I_{0}} = cP^k_{0,p},\quad B_{K_j,i}(\lambda_1,\lambda_2)|_{I_{p}} = cP^k_{p,p},
		%		\end{equation*}
	%	and
	%		\begin{eqnarray*}
		%			B_{K_j,i}(\lambda_1,\lambda_2)|_{I_{s-1}}-B_{K_j,i}(\lambda_1,\lambda_2)|_{I_{s}} = cP^k_{s,p},
		%		\end{eqnarray*}
	%		where 
	Set
	\begin{equation}\label{def:P}
		P_{p,s}(\lambda_1,\lambda_2): = s\lambda_1-(p-s)\lambda_2,\quad 0\leq s\leq p,
	\end{equation}
	and let $ I_s := (\xi_{j,s} , \xi_{j,s+1}) $ $ (0 \leq s \leq p-1) $ be the subintervals of $ K_j $. It is expected that the B-spline basis function $ B_{K_j,i}(\cdot) $ on each subinterval $ I_s $ will be uniquely presented in the following manner
	\begin{eqnarray}\label{eq,expression of Sij}
		B_{K_j,i}(\lambda_1,\lambda_2)|_{I_{s}}=\sum_{m=0}^{p}c_{s,m}\left( P_{p,m}(\lambda_1,\lambda_2)\right) ^p,\quad 0 \leq s \leq p-1,
	\end{eqnarray}
	where $ c_{s,m} $ are some constants to be determined.
	%		 By means of the expression \eqref{eq,expression of Sij}, it is easy to get the derivatives, i.e., 
	%		\begin{equation*}
		%			\dfrac{\mathrm{d}^{l} S_{K,i}(\lambda_1,\lambda_2)|_{I_{j}}}{\mathrm{d}x^l}= \frac{p!}{(p-l)!}\left(\frac{-p}{\xi{2}-\xi{1}}\right)^{l}\sum_{m=0}^{p}c_{m}P_{m+1,p}^{p-l}\quad \forall 1\leq l\leq p-1. 
		%		\end{equation*}
	Prior to exhibiting the representation \eqref{eq,expression of Sij} in question, it is necessary to undertake some preliminary preparation. 
	
	We first introduce the stirling numbers of the second kind \cite[Chapter 6]{Ronald1994}, for any nonnegative integers $ i,s $,
	\begin{eqnarray}\label{eq,stirling numbers of the second kind}
		\begin{Bmatrix}
			s \\
			i
		\end{Bmatrix} = \sum_{m=0}^{i}\frac{(-1)^{i-m}m^s}{(i-m)!m!},
	\end{eqnarray}
	which stands for the number of ways to partition a set of $ s $ things into $ i $ nonempty subsets. We have the following Lemma~\ref{lemma,stirling numbers}.
	\begin{lemma}\label{lemma,stirling numbers}
		For the  stirling numbers of the second kind \eqref{eq,stirling numbers of the second kind}, there holds
		\begin{eqnarray*}
			\begin{Bmatrix}s \\i\end{Bmatrix}=\left\{
			\begin{aligned}
				&0,& 0\leq s\leq i-1,\\
				&1,& s=i.
			\end{aligned}
			\right. 
		\end{eqnarray*}
	\end{lemma}
	Then, a special expansion of $ (\lambda_1+\lambda_2)^p $ is provided in  Lemma~\ref{lemma,expansion}.
	\begin{lemma}\label{lemma,expansion}
		We have the following expansion
		\begin{eqnarray}\label{eq,expansion}
			(\lambda_1+\lambda_2)^{p}=\sum_{m=0}^{p}\frac{(-1)^{p-m}}{(p-m)!m!}\left( P_{p,m}(\lambda_1,\lambda_2)\right)^{p} ,
		\end{eqnarray}
		where $ P_{p,m}(\lambda_1,\lambda_2) $ is given by \eqref{def:P}.
	\end{lemma}
	\begin{proof}
		%It is clear that the power functions $ P^k_{m+1,p} $ ($ 0\leq m\leq p $) are linearly independent. Then, $ (\lambda_1+\lambda_2)^{p} $ can be expanded in a unique manner as follows.
		%\begin{eqnarray*}
		%	(\lambda_1+\lambda_2)^{p}=\sum_{m=0}^{p}c_mP^k_{m+1,p}.
		%\end{eqnarray*}
		The desired expansion also reads
		\begin{eqnarray*}
			\sum_{i=0}^{p}C_p^i\lambda_1^i\lambda_2^{p-i}=\sum_{m=0}^{p}\frac{(-1)^{p-m}}{(p-m)!m!}\sum_{i=0}^{p}C_p^im^i(m-p)^{p-i}\lambda_1^i\lambda_2^{p-i},
		\end{eqnarray*}
		where $ C_p^i=\frac{p!}{i!(p-i)!} $. A comparison of the coefficients of $ \lambda_1^i\lambda_2^{p-i} $ on both sides of the above equation  yields
		\begin{eqnarray*}
			C_p^i = \sum_{m=0}^{p}\frac{(-1)^{p-m}}{(p-m)!m!}C_p^im^i(m-p)^{p-i},\quad i = 0,1,...,p.
		\end{eqnarray*}
		It is sufficient to demonstrate
		\begin{eqnarray*} \sum_{m=0}^{p}\frac{(-1)^{p-m}}{(p-m)!m!}m^i(m-p)^{p-i}=1,\quad i =  0,1,...,p.
		\end{eqnarray*}
		In fact, for $ i=0,1,...,p $, we have
		\begin{align*}
			\sum_{m=0}^{p}\frac{(-1)^{p-m}}{(p-m)!m!}m^i(m-p)^{p-i}&=\sum_{m=0}^{p}\frac{(-1)^{p-m}}{(p-m)!m!}m^i\sum_{j=0}^{p-i}C_{p-i}^{j}(-p)^{j}m^{p-i-j}\\
			&=\sum_{j=0}^{p-i}C_{p-i}^{j}(-p)^j\sum_{m=0}^{p}\frac{(-1)^{p-m}m^{p-j}}{(p-m)!m!}\\
			&=1,
		\end{align*} 
		where the last equation follows from  Lemma~\ref{lemma,stirling numbers}.
	\end{proof}
	
	The coefficients in the representation \eqref{eq,expression of Sij} are now ready for presentation. We end this subsection with the following Theorem~\ref{lemma,expression of Sij}.
	\begin{theorem}\label{lemma,expression of Sij}
		For	the B-spline element basis functions $ B_{K_j,i}(\cdot) $ $ (1\leq i\leq 2p) $, we have, for $ 1\leq i\leq p $,
		\begin{eqnarray}\label{eq,lemma,expression of Sij,1}
			B_{K_j,i}(\lambda_1,\lambda_2)|_{I_s} = \left\{
			\begin{aligned}
				&\sum_{m=s+1}^{i}\dfrac{(-1)^{i-m} i}{(i-m)!m!m^{p+1-i}}\left( P_{p,m}(\lambda_1,\lambda_2)\right)^{p},& 0\leq s\leq i-1,\\
				&0,& \mbox{otherwise},
			\end{aligned}
			\right. 
		\end{eqnarray}
		where $ P_{p,m}(\lambda_1,\lambda_2) $ is given by \eqref{def:P}, and  
		\begin{eqnarray}\label{eq,lemma,expression of Sij,2}
			B_{K_j,2p+1-i}(\lambda_1,\lambda_2)|_{I_{p-1-s}} = B_{K_j,i}(\lambda_2,\lambda_1)|_{I_s}\quad \forall 0\leq s \leq p-1.
		\end{eqnarray}
	\end{theorem}
	\begin{proof}
		For $ 1\leq i \leq p $, by the symmetry of knot vectors $ T_{K_j,i} $ and $ T_{K_j,2p+1-i} $ given in  \eqref{def:element knot basis 1} and \eqref{def:element knot basis 2},  it is evident that the related B-spline element basis functions satisfy the desired equation \eqref{eq,lemma,expression of Sij,2}. It remains to be stated that the equation \eqref{eq,lemma,expression of Sij,1} is true. The distinctive nature of B-spline basis functions enables us to restrict our verification to ensuring that the piecewise polynomials \eqref{eq,lemma,expression of Sij,1}  satisfy the following two properties.
		\begin{itemize}
			\item[(\textit{P1})] Each B-spline  basis function has  a minimal support with respect to a given degree and smoothness. That is, for $ 1\leq i\leq p $, the support of $ B_{K_j,i}(\cdot) $ is $ (\xi_{j,0},\xi_{j,i}) $, and  $ B_{K_j,i}(\cdot) $ has $ C^{p-1} $ smoothness acorss the knots $\xi_{j,1},...,\xi_{j,i}  $ and $ C^{i-2} $ smoothness acorss the knot $\xi_{j,0}  $.
			\item[(\textit{P2})] B-spline  basis functions on an open knot vector constitute a partition of unity. That is, $ \sum_{i=1}^{2p}B_{K_j,i}(x)=1 $ $ \forall x\in K_j $,  or, alternatively, $ \sum_{i=s+1}^{s+p+1}B_{K_j,i}(x)=1 $ on each subinterval $ I_s $.
		\end{itemize}
		
		It is easy to find that the support of $ B_{K_j,i}(\cdot) $ given by \eqref{eq,lemma,expression of Sij,1}  is $ (\xi_{j,0},\xi_{j,i}) $, and it  has $ C^{p-1} $ smoothness acorss the knots $\xi_{j,1},...,\xi_{j,i}  $. To show the property (\textit{P1}), taking the $ l $-th derivatives  at ($\lambda_1=1,\lambda_2=0$) gives
		\begin{eqnarray*}
			\dfrac{\mathrm{d}^{l} B_{K_j,i}(\cdot)}{\mathrm{d}x^l}\bigg|_{(\lambda_1=1,\lambda_2=0)}= \frac{p!i}{(p-l)!}\left(\frac{-p}{h_j}\right)^{l}\sum_{m=1}^{i}\frac{(-1)^{i-m}m^{i-1-l}}{(i-m)!m!}\quad \forall l\geq 0.
		\end{eqnarray*}
		Notice that, for $ i-1-l\geq 1 $, 
		\begin{eqnarray*}
			\sum_{m=1}^{i}\frac{(-1)^{i-m}m^{i-1-l}}{(i-m)!m!}=\sum_{m=0}^{i}\frac{(-1)^{i-m}m^{i-1-l}}{(i-m)!m!}=
			\begin{Bmatrix}
				i-1-l \\
				i
			\end{Bmatrix}.
		\end{eqnarray*}
		It follows from Lemma~\ref{lemma,stirling numbers} and $ i-1-l\geq 1 $ that
		\begin{eqnarray*}
			\dfrac{\mathrm{d}^{l} B_{K_j,i}(\cdot)}{\mathrm{d}x^l}\bigg|_{(\lambda_1=1,\lambda_2=0)} = 0\quad  \forall 0 \leq l\leq i-2.
		\end{eqnarray*}
		This indicates that $ B_{K_j,i}(\cdot) $ has $ C^{i-2} $ smoothness acorss the knot $\xi_{j,0}  $.
		
		We proceed to demonstrate that $ B_{K_j,i}(\cdot) $ $ (1\leq i\leq 2p) $ given by \eqref{eq,lemma,expression of Sij,1} and \eqref{eq,lemma,expression of Sij,2} satisfy the property (\textit{P2}). We start by proving the following result.
		\begin{equation}\label{eq,equi of eqs}
			\sum_{i=s+1}^{p}B_{K_j,i}(\lambda_1,\lambda_2)|_{I_s} = \sum_{m=s+1}^{p}\frac{(-1)^{p-m}}{m!(p-m)!}\left( P_{p,m}(\lambda_1,\lambda_2)\right)^{p}\,\,\forall0\leq s \leq p-1.
		\end{equation} 
		Regarding the left-hand side of the above equation, we have
		\begin{align*}
			\sum_{i=s+1}^{p}B_{K_j,i}(\lambda_1,\lambda_2)|_{I_s} &=\sum_{s+1\leq i\leq p,s+1\leq m\leq i}\dfrac{(-1)^{i-m} i}{(i-m)!m!m^{p+1-i}}\left( P_{p,m}(\lambda_1,\lambda_2)\right)^{p}\\
			&=\sum_{s+1\leq m\leq p,m\leq i\leq p}\dfrac{(-1)^{i-m} i}{(i-m)!m!m^{p+1-i}}\left( P_{p,m}(\lambda_1,\lambda_2)\right)^{p}\\
			&=\sum_{m=s+1}^{p}\sum_{i=m}^{p}\dfrac{(-1)^{i-m} i}{(i-m)!m!m^{p+1-i}}\left( P_{p,m}(\lambda_1,\lambda_2)\right)^{p}.
		\end{align*}
		Consequently, in order to prove \eqref{eq,equi of eqs}, it is necessary to show
		\begin{eqnarray}\label{eq,induce result}
			\sum_{i=m}^{p}\dfrac{(-1)^{i} i}{(i-m)!m^{p+1-i}}=\frac{(-1)^{p}}{(p-m)!}\quad \forall m\leq p.
		\end{eqnarray}
		We induce on $ p (\geq m) $. Obviously, \eqref{eq,induce result} holds for $ p=m $. Assume that \eqref{eq,induce result} is true for $ p=n(\geq m) $. Then, for $  p=n+1 $, we have
		\begin{align*}
			\sum_{i=m}^{n+1}\dfrac{(-1)^{i} i}{(i-m)!m^{n+2-i}}&=\frac{1}{m}\sum_{i=m}^{n}\dfrac{(-1)^{i} i}{(i-m)!m^{n+1-i}}+\dfrac{(-1)^{n+1} (n+1)}{(n+1-m)!m}\\
			&=\frac{1}{m}\left( \frac{(-1)^{n}}{(n-m)!} +\dfrac{(-1)^{n+1} (n+1)}{(n+1-m)!}\right)\\
			&=\frac{(-1)^{n+1}}{(n+1-m)!},
		\end{align*}
		which completes the induction step. Using   the symmetry relation \eqref{eq,lemma,expression of Sij,2} and the established equation \eqref{eq,equi of eqs} gives
		\begin{align}
			\sum_{i=1}^{s+1}B_{K_j,p+i}(\lambda_1,\lambda_2)|_{I_s} &= \sum_{i=p-s}^{p}B_{K_j,i}(\lambda_2,\lambda_1)|_{I_{p-1-s}}\nonumber\\
			&=\sum_{m=p-s}^{p}\frac{(-1)^{p-m}}{m!(p-m)!}\left( P_{p,m}(\lambda_2,\lambda_1)\right)^{p}\nonumber\\
			%				&=\sum_{m=p-s}^{p}\frac{(-1)^{p-m}(-1)^{p}}{m!(p-m)!}\left( P_{p,p-m}(\lambda_1,\lambda_2)\right)^{p}\nonumber\\
			&=\sum_{m=0}^{s}\frac{(-1)^{p-m}}{m!(p-m)!}\left( P_{p,m}(\lambda_1,\lambda_2)\right)^{p}.\label{eq,equi of eqs,2}
		\end{align}
		It follows from \eqref{eq,equi of eqs} and \eqref{eq,equi of eqs,2} that, for $ 0\leq s\leq p-1 $, 
		\begin{align*}
			\sum_{i=s+1}^{s+p+1}B_{K_j,i}(\lambda_1,\lambda_2)|_{I_s}&=\sum_{i=s+1}^{p}B_{K_j,i}(\lambda_1,\lambda_2)|_{I_s}+\sum_{i=1}^{s+1}B_{K_j,p+i}(\lambda_1,\lambda_2)|_{I_s}\\
			&=\sum_{m=0}^{p}\frac{(-1)^{p-m}}{m!(p-m)!}\left( P_{p,m}(\lambda_1,\lambda_2)\right)^{p} \\
			&=1,
		\end{align*}
		where the last equation is the consequence of Lemma~\ref{lemma,expansion} and the fact $ \lambda_1+\lambda_2=1 $.
	\end{proof}
	
	\section{Properties of element-based B-spline basis function spaces}
	\label{sec2}
	In this section, the focus will be on the properties of the element-based B-spline basis function space $ \mathcal{S}_{p}^{\boldsymbol{k}}(\Xi):= \text{span}\{\mathcal{B}_{p}(\Xi)\}$, where the basis functions $ \mathcal{B}_{p}(\Xi) $ are defined by B-spline element basis functions.
	
	Subsection~\ref{subsec2.1} discusses the Hermite interpolation property and employs element analysis to prove optimal error estimates, where the regularity $ \boldsymbol{k} $ can be optimal for degree $ p $ (compared with the interpolation error results in \cite{Beirao2011}). Subsection~\ref{subsec2.2} deals with the knot vector $ \Xi $ in \eqref{eq:knot vector} with $ r_2=...=r_{m-1}=1 $. The element-wise spline functions will be used as new basis functions to construct the $ C^{p-1} $ finite element space $ \mathcal{S}_p^{p-1}(\mathcal{T}_{h}) $, with the objective of later presenting certain advantages in applications of isogeometric analysis.
	\subsection{Hermite interpolation and optimal error estimates}\label{subsec2.1}
	\begin{definition}\label{def: pik}
		Given the element-based B-spline basis function space $ \mathcal{S}_{p}^{\boldsymbol{k}}(\Xi)$ and a sufficiently smooth function $ u(x)$,
		the Hermite interpolant $ \pi_{p}^{\boldsymbol{k}}u\in \mathcal{S}_{p}^{\boldsymbol{k}}(\Xi)  $ of $u(x)$ is defined such that, for $  j=1,...,m $,
		\begin{equation*}\label{eq:SE interpolation,condition}
			(\pi_{p}^{\boldsymbol{k}}u)^{(l)}(\xi_j)=u^{(l)}(\xi_j),\quad  l=0,1,...,p-r_j,
		\end{equation*}
		and 
		\begin{equation*}\label{eq:SE interpolation,condition,2}
			(\pi_{p}^{\boldsymbol{k}}u)^{(l)}(\xi_j^{-})=u^{(l)}(\xi_j^{-}),\quad (\pi_{p}^{\boldsymbol{k}}u)^{(l)}(\xi_j^{+})=u^{(l)}(\xi_j^{+}),\quad  l=p-r_j+1,...,p-1,
		\end{equation*}
		where $ f^{(l)}(\xi_j^{-}) $ denotes the $ l $-th left-hand derivative of $ f $ at $ \xi_j $, and $ f^{(l)}(\xi_j^{+}) $ denotes the $ l $-th right-hand  derivative of $ f $ at $ \xi_j $. In the case of $ r_2=...=r_{m-1}=1 $ for $ \Xi $,  the  Hermite interpolation operator $ \pi_{p}^{\boldsymbol{k}}\cdot $ is simply denoted by $ \pi_{p}^{p-1}\cdot $.
	\end{definition}
	\begin{theorem}
		The Hermite interpolation operator $ \pi_{p}^{\boldsymbol{k}}\cdot $ in Definition~\ref{def: pik} is well-defined.
	\end{theorem}
	\begin{proof}
		For any fixed $ j $  $ (1\leq j\leq m) $, it is observed that the total constraints  are $ p+r_j-1 $, which match the total degrees of freedom, that is, the number of the B-splines $ \mathcal{B}_{\xi_j,r_j} $ given by (see Fig.~\ref{fig: recurrence relation})
		\begin{align*}
			\{B_{K_{j-1},p+1},...,B_{K_{j-1},p+r_j-1},B_{p-r_j+1,1,j},...,B_{p-r_j+1,p-r_j+1,j},B_{K_{j},p-r_j+2},...,B_{K_{j},p} \}.
		\end{align*}
		Using Definition~\ref{def: pik} gives a square linear system of size $ p+r_j-1 $. Therefore, the uniqueness implies the existence. We are now left to show the uniqueness of the solution. If not, there exist $ p+r_j-1 $ constants $ L_i $,  $ R_i $ ($ 1\leq i\leq r_j-1  $) and $ M_i $  ($ 1\leq i\leq p-r_j+1  $), not all of which are zero, such that
		\begin{eqnarray*}
			\left\lbrace 
			\begin{aligned}
				&\sum_{i=1}^{p-r_j+1}M_iB^{(l)}_{p-r_j+1,i,j}(\xi_j)=0,\qquad l=0,1,...,p-r_j, \\
				&\sum_{i=1}^{r_j-1}L_iB^{(l)}_{K_{j-1},p+i}(\xi^{-}_j) + \sum_{i=1}^{p-r_j+1}M_iB^{(l)}_{p-r_j+1,i,j}(\xi^{-}_j)=0,\qquad l=p-r_j+1,...,p-1,\\
				&\sum_{i=1}^{r_j-1}R_iB^{(l)}_{K_{j},p-r_j+1+i}(\xi^{+}_j) + \sum_{i=1}^{p-r_j+1}M_iB^{(l)}_{p-r_j+1,i,j}(\xi^{+}_j)=0,\qquad l=p-r_j+1,...,p-1,\\
			\end{aligned}
			\right.
		\end{eqnarray*}
		As illustrated in Fig.~\ref{fig: recurrence relation}, the B-splines $ B_{p-r_j+1,i,j} $ ($ 1\leq i\leq p-r_j+1 $) are linear combinations of $ B_{K_{j-1},i} $ ($ p+r_j\leq i\leq 2p $) and $ B_{K_{j},i} $ ($ 1\leq i\leq p-r_j+1 $). Thus, the above system of equations also reads
		\begin{eqnarray}\label{eq:Li,Ri}
			\left\lbrace 
			\begin{aligned}
				&\sum_{i=1}^{p}L_iB^{(l)}_{K_{j-1},p+i}(\xi^{-}_j)=0,\qquad l=0,1,...,p-1,\\
				&\sum_{i=r_j-p}^{r_j-1}R_iB^{(l)}_{K_{j},p-r_j+1+i}(\xi^{+}_j)=0,\qquad l=0,1,...,p-1,
			\end{aligned}
			\right.
		\end{eqnarray}
		where $ L_i $ ($r_j \leq i\leq p  $) and $ R_i $ ($r_j-p \leq i\leq 0  $) are linear combinations of $ M_i $ ($ 1 \leq i\leq p-r_j+1  $), with the coefficients only related to $ h_{j-1}$ and $h_j $. For instance,
		\begin{eqnarray}\label{eq:Li,M_i}
			\begin{pmatrix}
				L_{r_j}\\
				L_{r_j+1}\\
				\vdots\\
				L_{p}\\
			\end{pmatrix}=   \left(\begin{array}{cccc}
				H_{1,1}& 0& \cdots & 0 \\
				H_{2,1}& H_{2,2} & \cdots& 0\\
				\vdots & \vdots  & \ddots & \vdots \\
				H_{p-r_j+1,1} &H_{p-r_j+1,2}&\cdots & H_{p-r_j+1,p-r_j+1}\\
			\end{array}
			\right)
			\begin{pmatrix}
				M_{1}\\
				M_{2}\\
				\vdots\\
				M_{p-r_j+1}\\
			\end{pmatrix},
		\end{eqnarray}
		where the elements $ H_{m,n}>0 $ for $ m,n:=1,2,..,p-r_j+1 $ and $ m\geq n $.   For  the first equation of \eqref{eq:Li,Ri}, it follows from the proof of  Theorem~\ref{lemma,expression of Sij} that
		\begin{align*}
			B^{(l)}_{K_{j-1},2p+1-i}(\xi^{-}_j) 
			& = \frac{(-1)^{p}p!i}{(p-l)!}\left(\frac{p}{h_{j-1}}\right)^{l}\sum_{m=1}^{i}\frac{(-1)^{i-m}m^{i-1-l}}{(i-m)!m!}\\
			&=
			\begin{cases}
				0,&0\leq l< i-1\leq p-1,\\
				\frac{(-1)^{p}p!i}{(p-l)!}\left(\frac{p}{h_{j-1}}\right)^{l},& 0\leq l= i-1\leq p-1.
			\end{cases}
		\end{align*}
		Therefore, the linear system given by the first equation of \eqref{eq:Li,Ri} admits only the trivial solution. Consequently, all coefficients $ M_i $ ($ 1\leq i\leq p-r_j+1  $) derived from the system \eqref{eq:Li,M_i}  must vanish. For the second equation of \eqref{eq:Li,Ri}, an analogous conclusion holds. This contradicts the non-vanishing of the constants $ L_i $,  $ R_i $ ($ 1\leq i\leq r_j-1  $) and $ M_i $  ($ 1\leq i\leq p-r_j+1  $). Thus, the operator $ \pi_{p}^{\boldsymbol{k}}\cdot $ is well-defined.
	\end{proof}
	
	For a function $ u(x) $, we say that $u(x)\in C^{\boldsymbol{k}}(\langle \Xi \rangle) $, $ \boldsymbol{k} = (k_1,...,k_{m}) $, if $ u(x) $ has $ C^{k_j} $ smoothness across the knot $ \xi_j $ for $ j=1,...,m $. 
	Define the element-wise  $ H^{p+1} $ space over $ \langle \Xi \rangle $ as 
	\begin{equation*}
		H_{h}^{p+1}(\langle \Xi \rangle): = \{u\in  C^{\boldsymbol{k}}(\Xi): \,\,\, u|_{K_j}\in H^{p+1}(K_j)\quad \forall j = 1,...,m   \}.
	\end{equation*}
	We have the following  optimal error estimates.
	\begin{theorem}\label{lemma: 1}
		For a function $ u(x)\in H_{h}^{p+1}(\langle \Xi \rangle)  $ and the  Hermite interpolant $ \pi_{p}^{\boldsymbol{k}}u\in \mathcal{S}_{p}^{\boldsymbol{k}}(\Xi)  $, we have the optimal local  approximation  error
		\begin{eqnarray}\label{eq: local interpolation error}
			|u-\pi_{p}^{\boldsymbol{k}}u|_{l,K_j} \leq Ch_j^{p+1-l}|u|_{p+1,K_j} \quad 
			\forall 0\leq l\leq p.
		\end{eqnarray}
		Consequently, 
		\begin{equation}\label{eq: global interpolation error}
			|u-\pi_{p}^{\boldsymbol{k}}u|_{l} \leq Ch^{p+1-l}|u|_{p+1,h} \quad 
			\forall 0\leq l\leq \min(\boldsymbol{k})+1.
		\end{equation}
		where $ |u|_{p+1,h} = (\sum_{j=1}^{m}|u|^2_{p+1,K_j})^{\frac{1}{2}} $.
	\end{theorem}
	\begin{proof}
		The global  approximation  error \eqref{eq: global interpolation error} follows directly the local  approximation  error \eqref{eq: local interpolation error}. To prove \eqref{eq: local interpolation error}, we first show 
		\begin{eqnarray}\label{eq: interpolation operator, exact}
			(\pi_{p}^{\boldsymbol{k}}w)(x)\equiv w(x)\quad \forall x\in K_j,
		\end{eqnarray}
		for any  polynomial $ w(x) $ of degree $ p $ on $ K_j $. By Definition~\ref{def: pik}, one gets that $ w $ and $ \pi_{p}^{\boldsymbol{k}}w $ have the same function values and the first to $ (p-1) $-th right-hand derivatives at $ \xi_j $ and left-hand derivatives at $ \xi_{j+1} $, which yields
		\begin{eqnarray}\label{eq:lemma 1,1}
			\left\{
			\begin{aligned}
				&\pi_{p}^{\boldsymbol{k}}w|_{I_0}-w = c_0\left( P_{p,0}(\lambda_1,\lambda_2)\right)^p,\\[1mm]
				&w - \pi_{p}^{\boldsymbol{k}}w|_{I_{p-1}} = c_{p}\left( P_{p,p}(\lambda_1,\lambda_2)\right)^p,
			\end{aligned}
			\right.
		\end{eqnarray}
		where $c_0,c_{p}  $ are two constants.
		Since $\pi_{p}^{\boldsymbol{k}}w\in C^{p-1}(K_j) $, it follows  that
		\begin{equation}\label{eq:lemma 1,2}
			\pi_{p}^{\boldsymbol{k}}w|_{I_{s}} - \pi_{p}^{\boldsymbol{k}}w|_{I_{s-1}} = c_{s}\left( P_{p,s}(\lambda_1,\lambda_2)\right)^p\quad \forall 1\leq s \leq p-1,
		\end{equation}
		where $c_s  $ are some  constants. Summing \eqref{eq:lemma 1,1} and \eqref{eq:lemma 1,2} yields
		\begin{eqnarray*}
			c_{0}\left( P_{p,0}(\lambda_1,\lambda_2)\right)^p+ c_{1}\left( P_{p,1}(\lambda_1,\lambda_2)\right)^p+...+ c_{p}\left( P_{p,p}(\lambda_1,\lambda_2)\right)^p=0.
		\end{eqnarray*}
		Clearly, $ c_{s}\left( P_{p,s}(\lambda_1,\lambda_2)\right)^p  $ $ (0\leq s\leq p) $ are linearly independent. Thus,
		\begin{eqnarray*}
			c_{0}=c_{1}=...=c_{p}=0,
		\end{eqnarray*}
		by which, we get the equation \eqref{eq: interpolation operator, exact}. The desired equation \eqref{eq: local interpolation error} follows by applying the Bramble-Hilbert lemma.
	\end{proof}
	
	\subsection{Reconstruction of basis via element-by-element representations}\label{subsec2.2}
	In this subsection, we consider the knot vector $ \Xi $ with $ r_2=...=r_{m-1}=1 $ and the related B-spline basis functions $ \mathcal{B}_p(\mathcal{T}_{h}) $ as well as the space $ \mathcal{S}_p^{p-1}(\mathcal{T}_{h}) $. The specific forms of the Hermite interpolant  $ \pi_{p}^{p-1}u $ of a function $ u\in C^{p-1}(\langle \Xi \rangle) $ will be provided. By representing $ \pi_{p}^{p-1}u $ element-by-element, the alternative spline basis functions of  $ \mathcal{S}_p^{p-1}(\mathcal{T}_{h}) $ are obtained.
	
	For each knot $ \xi_j $ ($ j=1,...,m $), we denote by $ B_{\xi_j,i}(x) $ ($ i=1,..., p$) the B-spline element basis functions in $ \mathcal{B}_{\xi_j,r_j} $. Then the Hermite interpolant  $ \pi_{p}^{p-1}u $ can be represented as follows 
	\begin{equation}\label{interpolation property,d1}
		(\pi_{p}^{p-1}u)(x)=\sum_{j=1}^{m}\sum_{i=1}^{p}c_{j,i}(u) B_{\xi_j,i}(x),
	\end{equation}
	where $c_{j,i}(u)$ are some constants to be determined with respect to $ u $. Recall the definition of $ B_{\xi_j,i}(x) $ given by  Theorem~\ref{theorem, generate Bspline} and Theorem~\ref{lemma,expression of Sij}. Then, for $ j=1,...,m $, solving 
	\begin{equation*}
		\sum_{i=1}^{p}c_{j,i}(u) B^{(l)}_{\xi_j,i}(\xi_j) = u^{(l)}(\xi_j), \quad l=0,1,...,p-1,
	\end{equation*}
	yields
	\begin{itemize}
		\item for $p=1$, 
		\begin{eqnarray}\label{interpolation, k1}
			c_{j,1}(u)=u(\xi_j),
		\end{eqnarray}
		\item for $p=2$, 
		\begin{eqnarray}\label{interpolation, k2}
			\left\{
			\begin{aligned}
				&c_{j,1}(u)=u(\xi_j)-\frac{h_{j-1}}{4}u^{(1)}(\xi_j),\\[1.5mm]
				&c_{j,2}(u)=u(\xi_j)+\frac{h_{j}}{4}u^{(1)}(\xi_j),
			\end{aligned}
			\right.
		\end{eqnarray}
		\item for $p=3$,
		\begin{eqnarray}\label{interpolation, k3}
			\left\{
			\begin{aligned}
				&c_{j,1}(u)=u(\xi_j)-\frac{h_{j-1}}{3}u^{(1)}(\xi_j)+\frac{h_{j-1}^{2}}{27}u^{(2)}(\xi_j),\\[1.5mm]
				&c_{j,2}(u)=u(\xi_j)+\frac{h_{j}-h_{j-1}}{9}u^{(1)}(\xi_j)-\frac{h_{j-1}h_{j}}{54}u^{(2)}(\xi_j),\\[1.5mm]
				&c_{j,3}(u)=u(\xi_j)+\frac{h_{j}}{3}u^{(1)}(\xi_j)+\frac{h_{j}^{2}}{27}u^{(2)}(\xi_j).
			\end{aligned}
			\right.
		\end{eqnarray} 
	\end{itemize}
	It is evident that, for $ p>1 $, the Hermite interpolation $ \pi_{p}^{p-1}u $ of $ \mathcal{S}_p^{p-1}(\mathcal{T}_{h}) $ differs from the interpolation of classical Hermite finite element spaces in two main ways. Firstly, the polynomial space degree is reduced for a same space regularity. Secondly, the interpolation coefficients $ c_{j,i}(u) $ are linear combinations of function values and derivative values, rather than being solely functions of these values.
	
	By the exact form \eqref{interpolation property,d1} of the Hermite interpolant  $ \pi_{p}^{p-1}u $ where the coefficients are given by \eqref{interpolation, k1}, \eqref{interpolation, k2} and \eqref{interpolation, k3}, representing it  element-by-element yields
	\begin{align*}
		(\pi_{p}^{p-1}u)(x)=&\sum_{j=1}^{m-1}(\pi_{p}^{p-1}u)(x)|_{K_j}\\
		=&\sum_{j=1}^{m-1}\sum_{i=1}^{p}\left( c_{j,i}(u) B_{\xi_j,i}(x)+c_{j+1,i}(u) B_{\xi_{j+1},i}(x)\right)|_{K_j}\\
		=&\sum_{j=1}^{m-1}\sum_{i=1}^{p}\left(u^{(i-1)}(\xi_j)\phi^{K_j}_{\xi_j,i}(x)+u^{(i-1)}(\xi_{j+1})\phi^{K_j}_{\xi_{j+1},i}(x)\right),
	\end{align*}
	where, using the B-spline element basis functions $ B_{K_j,i}(x) $ ($ i=1,...,2p $),
	\begin{itemize}
		\item for $p=1$, 
		\begin{eqnarray}\label{new basis, k1}
			\phi^{K_j}_{\xi_{j},1}(x)=B_{K_j,1}(x),\quad \phi^{K_j}_{\xi_{j+1},1}(x)=B_{K_j,2}(x),
		\end{eqnarray}
		\item for $p=2$, 
		\begin{eqnarray}\label{new basis, k2}
			\left\{
			\begin{aligned}
				&\phi^{K_j}_{\xi_{j},1}(x)=B_{K_j,1}(x)+B_{K_j,2}(x),\quad \phi^{K_j}_{\xi_{j},2}(x)=\frac{h_j}{4}B_{K_j,2}(x),\\[1.5mm]
				&\phi^{K_j}_{\xi_{j+1},1}(x)=B_{K_j,3}(x)+B_{K_j,4}(x),\quad \phi^{K_j}_{\xi_{j+1},2}(x)=-\frac{h_j}{4}B_{K_j,3}(x),
			\end{aligned}
			\right.
		\end{eqnarray}
		\item for $p=3$,
		\begin{eqnarray}\label{new basis, k3}
			\left\{
			\begin{aligned}
				&\qquad \phi^{K_j}_{\xi_{j},1}(x)=B_{K_j,1}(x)+B_{K_j,2}(x)+B_{K_j,3}(x),\\
				& \phi^{K_j}_{\xi_{j},2}(x)=h_j(\frac{1}{9}B_{K_j,2}(x)+\frac{1}{3}B_{K_j,3}(x)),\,\,\, \phi^{K_j}_{\xi_{j},3}(x)=\frac{h_j^{2}}{27}B_{K_j,3}(x),\\[1.5mm]
				&\qquad \phi^{K_j}_{\xi_{j+1},1}(x)=B_{K_j,4}(x)+B_{K_j,5}(x)+B_{K_j,6}(x),\\
				&\phi^{K_j}_{\xi_{j+1},2}(x)=-h_j(\frac{1}{3}B_{K_j,4}(x)+\frac{1}{9}B_{K_j,5}(x)),\,\,\,
				\phi^{K_j}_{\xi_{j+1},3}(x)=\frac{h_j^{2}}{27}B_{K_j,4}(x).\\[1.5mm]
			\end{aligned}
			\right.
		\end{eqnarray} 
	\end{itemize}
	It can be seen that the interpolant $ \pi_{p}^{p-1}u $, when expressed element-by-element, presents a very clean form.  For $ j=1,...,m $ and $ i = 1,...,p $, define
	\begin{eqnarray}\label{eq:new basis}
		B^*_{\xi_j,i}(x): = \left\{
		\begin{aligned}
			&\phi^{K_{j-1}}_{\xi_{j},i}, &x\in K_{j-1},\\
			&\phi^{K_{j}}_{\xi_{j},i}, &x\in K_{j},\\
			& 0,&  \mbox{otherwise}.
		\end{aligned}
		\right.
	\end{eqnarray}
	Then, the Hermite interpolant  $ \pi_{p}^{p-1}u $ has the following clean form expressed knot-by-knot
	\begin{equation}\label{eq:new basis interpolant}
		(\pi_{p}^{p-1}u)(x)=\sum_{j=1}^{m}\sum_{i=1}^{p}u^{(i-1)}(\xi_j)B^*_{\xi_j,i}(x).
	\end{equation}
	Denote 
	\begin{equation*}
		\mathcal{B}^*_p(\mathcal{T}_{h}):=\{B^*_{\xi_j,i}(x),\,\,\,j=1,...,m-1,\,\,i = 1,...,p\}.
	\end{equation*}
	It is not difficult to find that
	\begin{equation*}
		\text{span}\{	\mathcal{B}^*_p(\mathcal{T}_{h}) \} = \text{span}\{	\mathcal{B}_p(\mathcal{T}_{h}) \} = \mathcal{S}_p^{p-1}(\mathcal{T}_{h}).
	\end{equation*}
	\section{Element-based B-spline basis function spaces with applications to isogeometric analysis}
	\label{sec3}
	In this section, we consider the two-dimensional physical domain  obtained from the two-dimensional parameter domain via a geometric mapping in Computer-Aided Design (CAD). Within the context of isogeometric analysis (IgA), a discussion is then undertaken on the division of computations into those that are carried out on elements using element geometric mapping and the Hermite interpolation property.
	
	Subsection~\ref{subsec3.1} converts the geometric mapping in CAD to element geometric mappings by means of element analysis.
	Subsection~\ref{subsec3.2} deals with  the computation in IgA and divides it into element-by-element computations.
	\subsection{Geometric mapping with element analysis}
	\label{subsec3.1}
	We first introduce the parameter domain $ \hat{\Omega}:=\langle\Xi_{\xi}\rangle\otimes\langle\Xi_{\zeta}\rangle  $ in the $ \hat{x}\hat{y} $-plane, where
	\begin{equation*}
		\Xi_{\xi} =\{ r_1\xi_{1},\xi_{1,1},...,\xi_{1,p_1-1},r_2\xi_{2},...,r_{m_1-1}\xi_{m_1-1},\xi_{m_1-1,1},...,\xi_{m_1-1,p_1-1},r_{m_1}\xi_{m_1} \},
	\end{equation*}
	with  $ r_1=r_{m_1}=p_1+1 $, $ r_2=...=r_{m_1-1}\leq p_1 $, and
	\begin{equation*}
		\Xi_{\zeta} =\{ s_1\zeta_{1},\zeta_{1,1},...,\zeta_{1,p_2-1},s_2\zeta_{2},...,s_{m_2-1}\zeta_{m_2-1},\zeta_{m_2-1,1},...,\zeta_{m_2-1,p_2-1},s_{m_2}\zeta_{m_2} \},
	\end{equation*}
	with  $ s_1=s_{m_2}=p_2+1 $, $ s_2=...=s_{m_2-1}\leq p_2 $. The B-spline basis functions on  $ \Xi_{\xi} $ and $ \Xi_{\zeta} $ are denoted, respectively, by
	\begin{equation*}
		\mathcal{B}_{p_1}(\Xi_{\xi}) = \{\mathcal{B}_{\xi_1,r_1},\mathcal{B}_{\xi_2,r_2},...,\mathcal{B}_{\xi_{m_1},r_{m_1}}  \},
	\end{equation*}
	and
	\begin{equation*}
		\mathcal{B}_{p_2}(\Xi_{\zeta}) = \{\mathcal{B}_{\zeta_1,s_1},\mathcal{B}_{\zeta_2,s_2},...,\mathcal{B}_{\zeta_{m_2},s_{m_2}}  \},
	\end{equation*}
	where basis $ \mathcal{B}_{\xi_{j_1},r_{j_1}} $ $ (j_1=1,...,m_1) $ and $ \mathcal{B}_{\zeta_{j_2},s_{j_2}} $  $ (j_2=1,...,m_2) $ are  obtained from B-spline element basis functions by Theorem~\ref{theorem, generate Bspline}. There is a partition of $ \hat{\Omega} $ into rectangles:
	\begin{equation*}
		\mathcal{T}_{h}:= \{K_{j_1,j_2}=K_{\xi,j_1}\otimes K_{\zeta,j_2}, \,\, j_1=1,...,m_1-1,\,\, j_2=1,...,m_2-1\}, 
	\end{equation*}
	where $ K_{\xi,j_1} $ and $ K_{\zeta,j_2} $ are elements of $ \Xi_{\xi} $ and $ \Xi_{\zeta} $, and $ h $ is the maximum diameter of $ K_{j_1,j_2} $.
	The set of tensor-product B-splines basis functions on $ \Xi_{\xi}\otimes \Xi_{\zeta} $ is $ \mathcal{B}_{p_1,p_2}(\Xi_{\xi}\otimes \Xi_{\zeta}):= \mathcal{B}_{p_1}(\Xi_{\xi})\otimes\mathcal{B}_{p_2}(\Xi_{\zeta}) $, and the associated space is $ \mathcal{S}_{p_1,p_2}^{\boldsymbol{k_1},\boldsymbol{k_2}}(\Xi_{\xi}\otimes \Xi_{\zeta}) :=\text{span}\{\mathcal{B}_{p_1,p_2}(\Xi_{\xi}\otimes \Xi_{\zeta})\} $.
	
	Here and thereafter, we only consider the case $ p_1=p_2=p $, $ r_2=...=r_{m_1-1} =1 $ and $ s_2=...=s_{m_1-1} =1 $  to streamline the presentation. However, we shold keep in mind that the subsequent  discussion remains valid for the general case of knot vectors $ \Xi_{\xi} $ and $ \Xi_{\zeta} $. In this special case, we will also refer to  $ \mathcal{B}_{p_1,p_2}(\Xi_{\xi}\otimes \Xi_{\zeta})$ by  $ \mathcal{B}_{p,p}(\mathcal{T}_{h})$ and $ \mathcal{S}_{p_1,p_2}^{\boldsymbol{k_1},\boldsymbol{k_2}}(\Xi_{\xi}\otimes \Xi_{\zeta}) $ by $ \mathcal{S}_{p,p}^{p-1,p-1}(\mathcal{T}_{h}) $. We have 
	\begin{equation*}
		\mathcal{B}_{p,p}(\mathcal{T}_{h}):= \{ B_{\xi_{j_1},i_1}(\hat{x}) B_{\zeta_{j_2},i_2}(\hat{y}),\,\,\, j_a=1,...,m_a,\,\, i_a = 1,...,p,\,\, a =1,2 \},
	\end{equation*}
	where $ B_{\xi_{j_1},i_1}(\hat{x}) $ ($ i_1=1,...,p $) and $ B_{\zeta_{j_2},i_2}(\hat{y}) $ ($ i_2=1,...,p $)  are the basis of $ \mathcal{B}_{\xi_{j_1},r_{j_1}} $ and $ \mathcal{B}_{\zeta_{j_2},s_{j_2}} $, respectively. Recalling Subsection~\ref{subsec2.2}, another basis of $\mathcal{S}_{p,p}^{p-1,p-1}(\mathcal{T}_{h}) $ can be obtained  as follows
	\begin{equation}\label{def:new basis,2D}
		\{ B^*_{\xi_{j_1},i_1}(\hat{x}) B^*_{\zeta_{j_2},i_2}(\hat{y}),\,\,\, j_a=1,...,m_a,\,\, i_a = 1,...,p,\,\, a =1,2 \},
	\end{equation}
	where $ B^*_{\xi_{j_1},i_1}(\hat{x}) $ ($ i_1=1,...,p $) and $ B^*_{\zeta_{j_2},i_2}(\hat{y}) $ ($ i_2=1,...,p $) are defined by \eqref{eq:new basis} on each of the knot vectors $ \Xi_{\xi} $ and $ \Xi_{\zeta} $.
	
	In CAD, the physical domain $ \Omega $ in the $ xy $-plane is parameterized by a geometry mapping $ \boldsymbol{F}: \hat{\Omega}\rightarrow\Omega $, which is typically represented by rational B-splines. We consider this mapping generated directly by B-splines as follows 
	\begin{equation}\label{eq:mapping, B}
		\boldsymbol{F}(\hat{x},\hat{y}) = \sum_{j_1=1,j_2=1}^{m_1,m_2}\sum_{i_1=1,i_2=1}^{p,p}\boldsymbol{C}_{j_1,i_1;j_2,i_2}\boldsymbol{B}_{j_1,i_1;j_2,i_2}(\hat{x},\hat{y}),
	\end{equation}
	where $ \boldsymbol{B}_{j_1,i_1;j_2,i_2}(\hat{x},\hat{y}):= B_{\xi_{j_1},i_1}(\hat{x}) B_{\zeta_{j_2},i_2}(\hat{y})$, and the coefficients $ \boldsymbol{C}_{j_1,i_1;j_2,i_2}\in \mathbb{R}^{2} $ are referred as control points for the physical domain $ \Omega $. In the majority of cases, the non-interpolatory nature of the B-spline basis functions  precludes the interpretation of the control point values. However, as can be seen from Section~\ref{subsec2.2}, we can get some interpretations of these control points by the Hermite interpolation.
	Furthermore, using the new basis \eqref{def:new basis,2D} for the mapping \eqref{eq:mapping, B} enables us to obtain
	\begin{equation}\label{eq:mapping, B,*}
		\boldsymbol{F}(\hat{x},\hat{y}) = \sum_{j_1=1,j_2=1}^{m_1,m_2}\sum_{i_1=1,i_2=1}^{p,p}\boldsymbol{U}_{j_1,i_1;j_2,i_2}\boldsymbol{B}^{*}_{j_1,i_1;j_2,i_2}(\hat{x},\hat{y}),
	\end{equation}
	where $ \boldsymbol{B}^*_{j_1,i_1;j_2,i_2}(\hat{x},\hat{y}):= B^*_{\xi_{j_1},i_1}(\hat{x}) B^*_{\zeta_{j_2},i_2}(\hat{y})$. For each point $ (\xi_{j_1},\zeta_{j_2}) $ on $ \hat{\Omega} $, the $ p^2 $ coefficients $ \boldsymbol{U}_{j_1,i_1;j_2,i_2} $ ($ i_a = 1,...,p,\, a =1,2 $) can be uniquely obtained from the $ p^2 $ coefficients $ \boldsymbol{C}_{j_1,i_1;j_2,i_2}$ ($ i_a = 1,...,p,\, a =1,2 $). Taking $ p=2 $ as an example,  using the interpolation property \eqref{interpolation, k2} yields
	{\small \begin{eqnarray*}
			\left\{
			\begin{aligned}
				&\boldsymbol{C}_{j_1,1;j_2,1} = \boldsymbol{U}_{j_1,1;j_2,1}-\frac{h_{\xi,j_1-1}}{4}\boldsymbol{U}_{j_1,2;j_2,1}-\frac{h_{\zeta,j_2-1}}{4}\boldsymbol{U}_{j_1,1;j_2,2}+\frac{h_{\xi,j_1-1}h_{\zeta,j_2-1}}{16}\boldsymbol{U}_{j_1,2;j_2,2},\\
				&\boldsymbol{C}_{j_1,2;j_2,1} = \boldsymbol{U}_{j_1,1;j_2,1}+\frac{h_{\xi,j_1}}{4}\boldsymbol{U}_{j_1,2;j_2,1}-\frac{h_{\zeta,j_2-1}}{4}\boldsymbol{U}_{j_1,1;j_2,2}-\frac{h_{\xi,j_1}h_{\zeta,j_2-1}}{16}\boldsymbol{U}_{j_1,2;j_2,2},\\
				&\boldsymbol{C}_{j_1,1;j_2,2} = \boldsymbol{U}_{j_1,1;j_2,1}-\frac{h_{\xi,j_1-1}}{4}\boldsymbol{U}_{j_1,2;j_2,1}+\frac{h_{\zeta,j_2}}{4}\boldsymbol{U}_{j_1,1;j_2,2}-\frac{h_{\xi,j_1-1}h_{\zeta,j_2}}{16}\boldsymbol{U}_{j_1,2;j_2,2},\\
				&\boldsymbol{C}_{j_1,2;j_2,2} = \boldsymbol{U}_{j_1,1;j_2,1}+\frac{h_{\xi,j_1}}{4}\boldsymbol{U}_{j_1,2;j_2,1}+\frac{h_{\zeta,j_2}}{4}\boldsymbol{U}_{j_1,1;j_2,2}+\frac{h_{\xi,j_1}h_{\zeta,j_2}}{16}\boldsymbol{U}_{j_1,2;j_2,2}.
			\end{aligned}
			\right.
	\end{eqnarray*}}
	where $ h_{\xi,j_1} $ and $ h_{\zeta,j_2} $ are the lengths of elements $ K_{\xi,j_1} $ and $ K_{\zeta,j_2} $, respectively. 	It is obvious that $ \boldsymbol{U}_{j_1,1;j_2,1} $ stands for the exact location  in the physical domain $ \Omega $, which is mapped from $ (\xi_{j_1},\zeta_{j_2}) $ of the parameter domain $ \hat{\Omega} $ through $ \boldsymbol{F} $. In view of the element-wise definition of the basis functions $ B^*_{\xi_{j_1},i_1}$ and $B^*_{\zeta_{j_2},i_2}$, as displayed in \eqref{eq:new basis}, the geometric mapping $  \boldsymbol{F}  $ can be reduced to a number of element geometric mappings as follows:
	\begin{equation}\label{eq:mapping, B,*,element}
		\boldsymbol{F}(\hat{x},\hat{y}) = \sum_{j_1=1,j_2=1}^{m_1-1,m_2-1}\boldsymbol{F}_{K_{j_1,j_2}}(\hat{x},\hat{y}),
	\end{equation}
	with
	\begin{equation*}
		\boldsymbol{F}_{K_{j_1,j_2}}(\hat{x},\hat{y}):=\sum_{n_1=0,n_2=0}^{1,1}\sum_{i_1=1,i_2=1}^{p,p}\boldsymbol{U}_{j_1+n_1,i_1;j_2+n_2,i_2}\phi^{K_{\xi,j_1}}_{\xi_{j_1+n_1},i_1}(\hat{x})\phi^{K_{\zeta,j_2}}_{\zeta_{j_2+n_2},i_2}(\hat{y})
	\end{equation*}
	where functions $ \phi^{K_{\xi,j_1}}_{\xi_{j_1},i_1}$ and $\phi^{K_{\zeta,j_2}}_{\zeta_{j_2},i_2} $ are linear combinations of B-spline element basis functions in each direction, e.g., see 
	\eqref{new basis, k1}, \eqref{new basis, k2} and \eqref{new basis, k3}.
	
	In the end of this subsection, we briefly outline the refinement. In CAD, the refinement of physical domains is often done by updating control points through the process of knot insertion. It is a general observation that, in most cases, only the newly control points at the corners remain the same as the original control points.   Note that the uniform refinements of the two knot vectors $ \Xi_{\xi} $ and $ \Xi_{\zeta} $ result in an identical mesh to that produced by the uniform refinement of the mesh $ \mathcal{T}_h $. Using the equivalent mapping \eqref{eq:mapping, B,*}, we can consider the refinement of the mesh $ \mathcal{T}_h $ to update the coefficients $ \boldsymbol{U}_{j_1,i_1;j_2,i_2} $ associated with the vertices. Due to the interpolatory nature of the original coefficients, they remain unchanged in the refinement. Consequently, it is only necessary to provide the coefficients associated with the newly generated vertices following the refinement of the mesh $ \mathcal{T}_h $. This process is straightforward and we omit the discussion here.
	\begin{remark}
		We assume that the physical domain is parameterized directly by B-splines, although it is well known that rational B-splines are employed in CAD. Indeed, the discussion can be extended to the case of using rational B-splines, since the Hermite interpolation can still be well defined, but in a more complex form. In addition, non-rational B-splines are sufficient to represent the geometry in many cases. 
	\end{remark}
	\subsection{Computations in isogeometric analysis via elements}
	\label{subsec3.2}
	In isogeometric analysis, the approximation space $ \mathcal{V}_h $ of the space $ \mathcal{V} $ on the physical domain $ \Omega $ is defined by
	\begin{equation*}
		\mathcal{V}_h:=\text{span}\{\boldsymbol{D}_{j_1,i_1;j_2,i_2}:=\boldsymbol{B}_{j_1,i_1;j_2,i_2}\circ\boldsymbol{F}^{-1}\}_{j_1=1,i_1=1;j_2=1,i_2=1}^{m_1,p;m_2,p},
	\end{equation*}
	where the geometric mapping $ \boldsymbol{F} $ is given by \eqref{eq:mapping, B}, which can be reduced to element geometric mappings \eqref{eq:mapping, B,*,element}.
	We consider the use of the approximation space  as follows
	\begin{equation}\label{approxiamtion space,B,*}
		\mathcal{V}^*_h:=\text{span}\{\boldsymbol{D}^*_{j_1,i_1;j_2,i_2}:=\boldsymbol{B}^*_{j_1,i_1;j_2,i_2}\circ\boldsymbol{F}^{-1}\}_{j_1=1,i_1=1;j_2=1,i_2=1}^{m_1,p;m_2,p}.
	\end{equation}
	In the one-dimensional case, the interpolation form \eqref{eq:new basis interpolant} exhibits a more simplified expression than the interpolation form \eqref{interpolation property,d1}, both in terms of the element-by-element representations of the basis functions and the coefficients. This will facilitate the subsequent description, where the same discussion makes no difference for $ \mathcal{V}_h $ instead.
	
	For the exact solution $ u(x,y) $ of a model problem, the approximate solution $ u_h(x,y) $ is written as
	\begin{equation}\label{u_h,B,*}
		u_h(x,y)=\sum_{j_1=1,j_2=1}^{m_1,m_2}\sum_{i_1=1,i_2=1}^{p,p}u_{j_1,i_1;j_2,i_2}(\boldsymbol{B}^*_{j_1,i_1;j_2,i_2}\circ\boldsymbol{F}^{-1})(x,y),
	\end{equation}
	with unknown coefficients $ u_{j_1,i_1;j_2,i_2} $. In considering the relationship between the exact solution and the numerical solution (sometimes helping us to deal with the boundary conditions of the model problem),  we draw 
	\begin{equation*}
		u(\boldsymbol{F}(\hat{x},\hat{y}) ) \approx \sum_{j_1=1,j_2=1}^{m_1,m_2}\sum_{i_1=1,i_2=1}^{p,p}u_{j_1,i_1;j_2,i_2}\boldsymbol{B}^*_{j_1,i_1;j_2,i_2}(\hat{x},\hat{y}).
	\end{equation*}
	For each point $ (\xi_{j_1},\zeta_{j_2}) $ on $ \hat{\Omega} $,  we obtain
	\begin{equation}\label{eq: boundary imposition,1}
		u(\boldsymbol{F}(\xi_{j_1},\zeta_{j_2}) ) \approx u_{j_1,1;j_2,1},
	\end{equation}
	and 
	\begin{equation}\label{eq: boundary imposition,2}
		\frac{\partial u(\boldsymbol{F}(\hat{x},\hat{y}) )}{\partial \hat{x}}\big|_{(\hat{x},\hat{y})=(\xi_{j_1},\zeta_{j_2})}=\nabla u\cdot\frac{\partial \boldsymbol{F}}{\partial \hat{x}}\big|_{(\hat{x},\hat{y})=(\xi_{j_1},\zeta_{j_2})}=\nabla u\cdot\boldsymbol{U}_{j_1,2;j_2,1} \approx u_{j_1,2;j_2,1},
	\end{equation}
	where the left-hand side denotes the derivative of $ u $ at $ \boldsymbol{F}(\xi_{j_1},\zeta_{j_2}) $ along the image of $\hat{y}=\zeta_{j_2}  $ through $ \boldsymbol{F} $. Other relationships can also be obtained  from the chain rule.
	
	Given the possibility of constructing the geometric mapping and the B-spline basis functions element by element, the numerical calculations in isogeometric analysis can be performed by element-by-element calculations in the same way as in finite element analysis. We take the following inner product as an example.
	\begin{equation*}
		\int_{\Omega}\nabla\boldsymbol{D}^*_{j_1,i_1;j_2,i_2}\cdot\nabla\boldsymbol{D}^*_{j_1+1,i_3;j_2,i_4} \mathrm{d} \boldsymbol{x},
	\end{equation*}
	with $ \mathrm{d} \boldsymbol{x}:=\mathrm{d}x\mathrm{d}y $. Letting $ \mathrm{d} \hat{\boldsymbol{x}}:=\mathrm{d}\hat{x}\mathrm{d}\hat{y} $, we have
	\begin{align*}
		&\int_{\Omega}\nabla\boldsymbol{D}^*_{j_1,i_1;j_2,i_2}\cdot\nabla\boldsymbol{D}^*_{j_1+1,i_3;j_2,i_4} \mathrm{d} \boldsymbol{x}\\
		=&\int_{\hat{\Omega}}(\hat{\nabla}\boldsymbol{F}^{-T}\hat{\nabla}\boldsymbol{B}^*_{j_1,i_1;j_2,i_2})\cdot(\hat{\nabla}\boldsymbol{F}^{-T}\hat{\nabla}\boldsymbol{B}^*_{j_1+1,i_3;j_2,i_4})|\det \hat{\nabla}\boldsymbol{F}| \mathrm{d} \hat{\boldsymbol{x}}\\
		=&\int_{\hat{\Omega}}\hat{\nabla}\boldsymbol{B}^{*T}_{j_1,i_1;j_2,i_2}\mathcal{H}(\boldsymbol{F})\hat{\nabla}\boldsymbol{B}^*_{j_1+1,i_3;j_2,i_4} \mathrm{d} \hat{\boldsymbol{x}}
	\end{align*}
	where $ \hat{\nabla} $ denotes the gradient operator in the parameter domain $ \hat{\Omega} $, and 
	\begin{equation*}
		\mathcal{H}(\boldsymbol{F}):=\hat{\nabla}\boldsymbol{F}^{-1}\hat{\nabla}\boldsymbol{F}^{-T}|\det \hat{\nabla}\boldsymbol{F}|.
	\end{equation*}
	By the element-by-element representations of the mapping $ \boldsymbol{F} $ in \eqref{eq:mapping, B,*,element} and the spline basis functions $ \boldsymbol{B}^{*T}_{j_1,i_1;j_2,i_2} $ in \eqref{eq:new basis} for each direction, we obtain
	\begin{align*}
		&\int_{\hat{\Omega}}\hat{\nabla}\boldsymbol{B}^{*T}_{j_1,i_1;j_2,i_2}\mathcal{H}(\boldsymbol{F})\hat{\nabla}\boldsymbol{B}^*_{j_1+1,i_3;j_2,i_4} \mathrm{d} \hat{\boldsymbol{x}}\\
		=&\int_{K_{j_1,j_2}+K_{j_1,j_2-1}}\hat{\nabla}\boldsymbol{B}^{*T}_{j_1,i_1;j_2,i_2}\mathcal{H}(\boldsymbol{F})\hat{\nabla}\boldsymbol{B}^*_{j_1+1,i_3;j_2,i_4} \mathrm{d} \hat{\boldsymbol{x}}\\
		=&\int_{K_{j_1,j_2}}\hat{\nabla}(\phi^{K_{\xi,j_1}}_{\xi_{j_1},i_1}(\hat{x})\phi^{K_{\zeta,j_2}}_{\zeta_{j_2},i_2}(\hat{y}))^{T}\mathcal{H}(\boldsymbol{F}_{K_{j_1,j_2}})\hat{\nabla}(\phi^{K_{\xi,j_1}}_{\xi_{j_1+1},i_3}(\hat{x})\phi^{K_{\zeta,j_2}}_{\zeta_{j_2},i_4}(\hat{y})) \mathrm{d} \hat{\boldsymbol{x}}\\
		&+\int_{K_{j_1,j_2-1}}\hat{\nabla}(\phi^{K_{\xi,j_1}}_{\xi_{j_1},i_1}(\hat{x})\phi^{K_{\zeta,j_2-1}}_{\zeta_{j_2},i_2}(\hat{y}))^{T}\mathcal{H}(\boldsymbol{F}_{K_{j_1,j_2-1}})\hat{\nabla}(\phi^{K_{\xi,j_1}}_{\xi_{j_1+1},i_3}(\hat{x})\phi^{K_{\zeta,j_2-1}}_{\zeta_{j_2},i_4}(\hat{y})) \mathrm{d} \hat{\boldsymbol{x}}.
	\end{align*}
	It is evident that the computations on all elements follow the same pattern, irrespective of the variable lengths $ h_{\xi,j_1} $ and $ h_{\zeta,j_2} $. Consequently, the table look-up operation  of matrix assembly in isogeometric numerical simulations can be performed for non-uniform knots, not only for uniform knots, e.g., in the literature \cite{Antolin2015,Calabro2017,Pan2020}.
	\section{Numerical experiments}\label{sec4}
%	In this section,  the numerical performance of element-based B-spline basis function spaces employed in isogemetric analysis  is displayed.  
%	
%	The test examples are taken as Possion equation and Biharmonic equation, with respect to Subsection~\ref{subsec4.1} and Subsection~\ref{subsec4.2}.
This section investigates the numerical behavior of element-based B-spline basis function spaces within isogeometric analysis. We resolve the lack of interpolatory properties in conventional control points by constructing interpolatory points via the Hermite scheme, enabling element-wise representation of the geometric mapping $ \boldsymbol{F} $. In this way, such element-wise parametrization of IgA facilitates matrix assembly procedures analogous to those of FEM regardless of knot uniformity.

In Subsection~\ref{subsec5.1}, we demonstrate reduced points updates during uniform refinements for 2D CAD models and identify superconvergence of solution coefficients toward exact function/derivative values; In Subsection~\ref{subsec5.2}, computational superiority over conventional implementations is verified through two 3D CAD models with non-uniform parameterizations.

% focusing on two key aspects: (1) For 2D CAD models, we resolve the lack of interpolatory properties in conventional control points by constructing interpolatory points via the Hermite scheme, enabling element-wise representation of the geometric mapping $ \boldsymbol{F} $. In Subsection~\ref{subsec5.1}, numerical tests on uniform grids for Poisson and biharmonic problems reveal solution coefficients achieving superconvergence to exact values/derivatives. (2) For 3D CAD models with non-uniform grids of the parameter domain, element-wise representation of $ \boldsymbol{F} $ ensures compatibility with finite element matrix assembly patterns, thereby streamlining computational procedures. Subsection~\ref{subsec5.1} provides a concise demonstration of the enhanced computational efficiency while preserving exact geometric representations.
\subsection{Numerical examples in two dimension}\label{subsec5.1}
	\subsubsection{2D Possion equation}
	\label{subsec4.1}
	\begin{figure}[htbp]%% placement specifier
		%% Use \includegraphics command to insert graphic files. Place graphics files in 
		%% working directory.
		\centering%% For centre alignment of image.
		\subfigure[Control net and physical domain $ \Omega $]{
			\begin{minipage}{.45\textwidth}
				\centering
				\includegraphics[width=180pt]{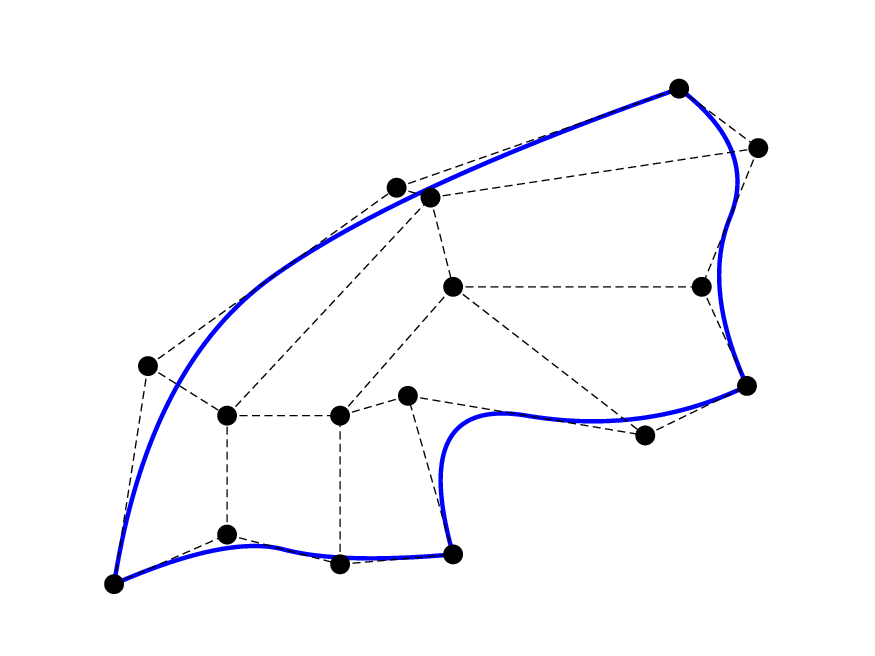}
				\label{fig2.1}
			\end{minipage}
		}
		\subfigure[Exact solution]{
			\begin{minipage}{.45\textwidth}
				\centering
				\includegraphics[width=180pt]{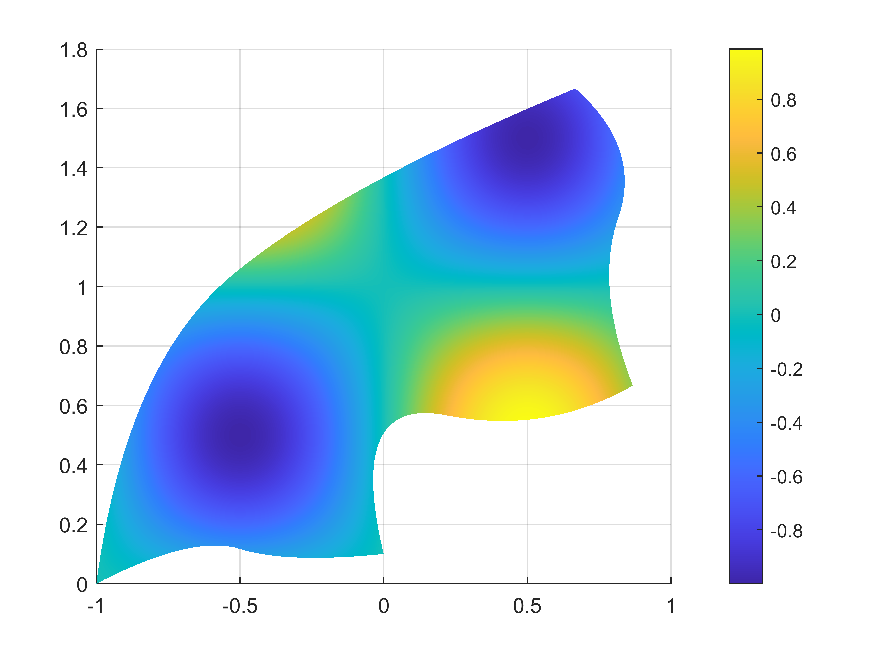}
				\label{fig2.2}
			\end{minipage}
		}
		%% Use \caption command for figure caption and label.
		\caption{The physical domain $ \Omega $ and exact solution  of Example~\ref{example,2}.}\label{fig2}
		%% https://en.wikibooks.org/wiki/LaTeX/Importing_Graphics#Importing_external_graphics
	\end{figure}
	\begin{table}[htbp]%% placement specifier
		%% Use tabular environment to tag the tabular data.
		%% https://en.wikibooks.org/wiki/LaTeX/Tables#The_tabular_environment
		\centering%% For centre alignment of tabular.
		{\footnotesize 	\begin{tabular}{c c c}%% Table column specifiers
				%% Tabular cells are separated by &
				\hline
				$i_1,i_2 $ & $ \boldsymbol{C}_{1,i_1;1,i_2} $ &$ \boldsymbol{U}_{1,i_1;1,i_2} $ \\
				\hline
				1,1 & (-1.00,0.00) & (-1.00, 0.00)  \\ %% A tabular row ends with \\
				2,1 & (-0.67,0.17) & ( 0.40, 2.93)\\
				1,2 & (-0.90,0.73) & ( 1.33, 0.67)\\
				2,2 & (-0.67,0.57) & (-1.60,-5.33)\\
				\hline
			\end{tabular}\,\,\,
			\begin{tabular}{c c c}%% Table column specifiers
				%% Tabular cells are separated by &
				\hline
				$i_1,i_2 $ & $ \boldsymbol{C}_{2,i_1;1,i_2} $ &$ \boldsymbol{U}_{2,i_1;1,i_2} $ \\
				\hline
				1,1 & (-0.33,0.07) & ( 0.00, 0.10) \\
				2,1 & ( 0.00,0.10) & (-0.53, 2.13)  \\
				1,2 & (-0.33,0.57) & ( 1.33, 0.13)\\
				2,2 & (-1.33,0.63) & (-2.13, 0.53)\\
				\hline
			\end{tabular}\\[2mm]
			\begin{tabular}{c c c}%% Table column specifiers
				%% Tabular cells are separated by &
				\hline
				$i_1,i_2 $ & $ \boldsymbol{C}_{1,i_1;2,i_2} $ &$ \boldsymbol{U}_{1,i_1;2,i_2} $ \\
				\hline
				1,1 & (-0.17,1.33) & ( 0.67, 1.67)  \\ %% A tabular row ends with \\
				2,1 & (-0.07,1.30) & ( 3.33, 1.33)\\
				1,2 & ( 0.67,1.67) & ( 0.93,-0.80)\\
				2,2 & ( 0.90,1.47) & ( 2.13,-2.67)\\
				\hline
			\end{tabular}\,\,\,
			\begin{tabular}{c c c}%% Table column specifiers
				%% Tabular cells are separated by &
				\hline
				$i_1,i_2 $ & $ \boldsymbol{C}_{2,i_1;2,i_2} $ &$ \boldsymbol{U}_{2,i_1;2,i_2} $ \\
				\hline
				1,1 & ( 0.00,1.00) & ( 0.87, 0.67) \\
				2,1 & ( 0.57,0.50) & ( 1.20, 0.67)  \\
				1,2 & ( 0.73,1.00) & ( 0.53,-1.33)\\
				2,2 & ( 0.87,0.67) & (-6.93, 2.67)\\
				\hline
		\end{tabular}}
		%% Use \caption command for table caption and label.
		\caption{The control points $ \boldsymbol{C}_{j_1,i_1;j_2,i_2} $ for the control net in Fig.~\ref{fig2}, and the related interpolatory points $ \boldsymbol{U}_{j_1,i_1;j_2,i_2} $ derived from the Hermite interpolation property.}
		\label{table,example2}
	\end{table}
	\begin{example}\label{example,2}
		In this example, we consider the parameter domain $ \hat{\Omega}:=\langle\Xi_{\xi}\rangle\otimes\langle\Xi_{\zeta}\rangle  $ in the $ \hat{x}\hat{y} $-plane with
		\begin{equation*}
			\Xi_{\xi} =\Xi_{\zeta}=\{ 0,0,0,\frac{1}{2},1,1,1\}.
		\end{equation*}
		The physical domain $ \Omega$ is obtained through the geometric mapping
		\begin{equation*}
			\boldsymbol{F}(\hat{x},\hat{y}) = \sum_{j_1=1,j_2=1}^{2,2}\sum_{i_1=1,i_2=1}^{2,2}\boldsymbol{C}_{j_1,i_1;j_2,i_2}\boldsymbol{B}_{j_1,i_1;j_2,i_2}(\hat{x},\hat{y}),
		\end{equation*}
		where the control points $ \boldsymbol{C}_{j_1,i_1;j_2,i_2} $ are given  in Table~\ref{table,example2}. The control net and the physical domain are illustrated in Fig.~\ref{fig2.1}.
		In addition, using the new basis \eqref{def:new basis,2D} for the mapping \eqref{eq:mapping, B} gives 
		\begin{equation*}
			\boldsymbol{F}(\hat{x},\hat{y}) = \sum_{j_1=1,j_2=1}^{2,2}\sum_{i_1=1,i_2=1}^{2,2}\boldsymbol{U}_{j_1,i_1;j_2,i_2}\boldsymbol{B}^{*}_{j_1,i_1;j_2,i_2}(\hat{x},\hat{y}),
		\end{equation*}
		where the points $ \boldsymbol{U}_{j_1,i_1;j_2,i_2} $ with interpolatory properties are given in Table~\ref{table,example2}.	We solve the 2D Possion equation 
		\begin{eqnarray*}
			\left\{
			\begin{aligned}
				&-\Delta u =f\quad \mbox{in $\Omega$},\\
				&u=g \quad \mbox{on $\partial\Omega$},
			\end{aligned}
			\right.
		\end{eqnarray*}  
		with the following exact solution:
		\begin{equation*}
			u(x,y) = \sin(\pi x)\sin{(\pi y)},
		\end{equation*}
		where the functions $ f $ and $ g $ are given by the exact solution. The image of the exact solution is presented in Fig.~\ref{fig2.2}. 
	\end{example}
	\begin{figure}[htbp]%% placement specifier
		%% Use \includegraphics command to insert graphic files. Place graphics files in 
		%% working directory.
		\centering%% For centre alignment of image.
		\subfigure[$ \Omega:=\boldsymbol{F}(\hat{\Omega}) $]{
			\begin{minipage}{.3\textwidth}
				\centering
				\includegraphics[width=130pt]{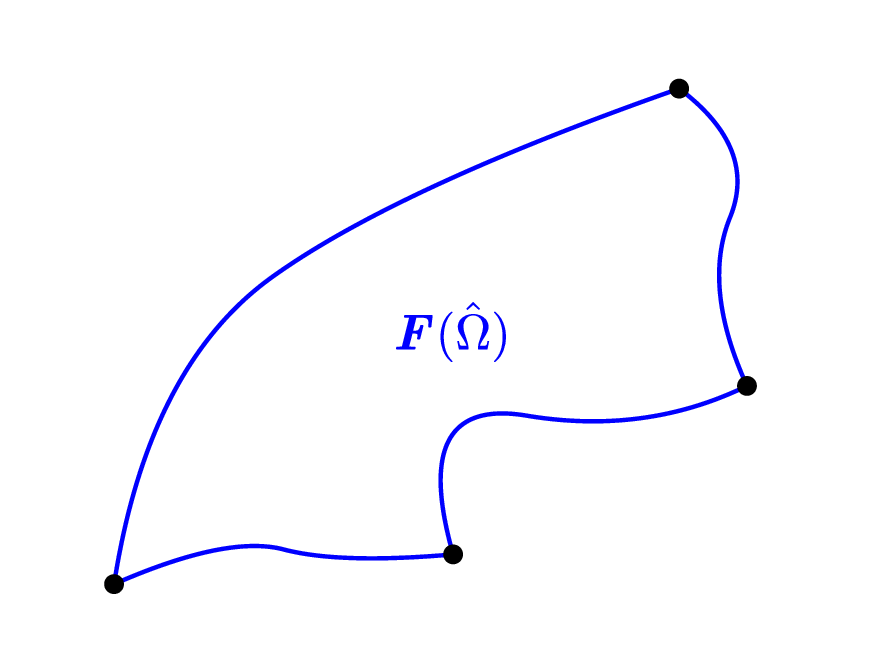}
				\label{fig5.1}
			\end{minipage}
		}
		\subfigure[$ \boldsymbol{F}(\mathcal{T}^{(1)}_{h,ur}) $]{
			\begin{minipage}{.3\textwidth}
				\centering
				\includegraphics[width=130pt]{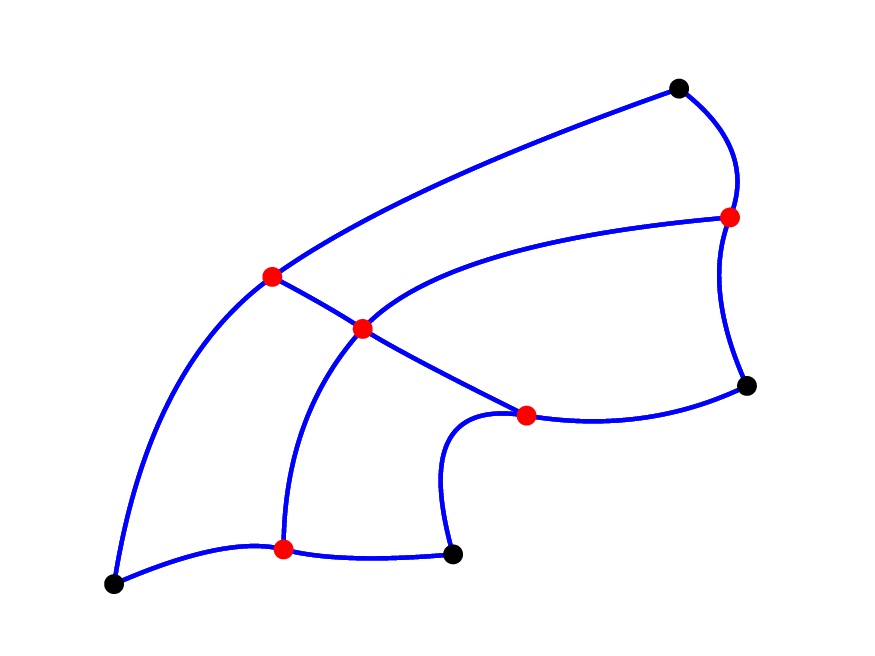}
				\label{fig5.2}
			\end{minipage}
		}
		\subfigure[$ \boldsymbol{F}(\mathcal{T}^{(2)}_{h,ur}) $]{
			\begin{minipage}{.3\textwidth}
				\centering
				\includegraphics[width=130pt]{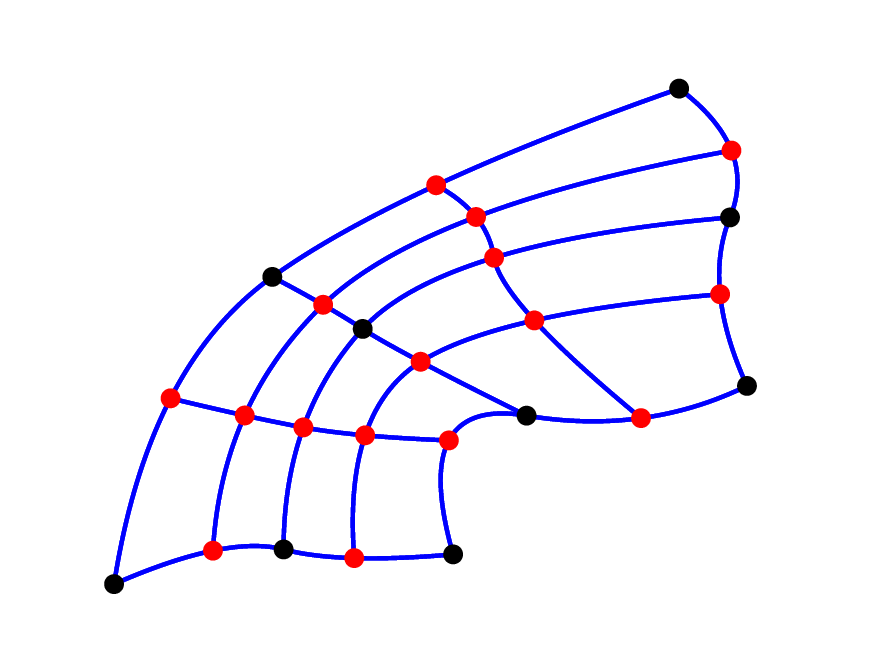}
				\label{fig5.3}
			\end{minipage}
		}
		%% Use \caption command for figure caption and label.
		\caption{The images of the meshes $ \mathcal{T}^{(n)}_{h,ur} $ through the geometric mapping $ \boldsymbol{F} $ in Example~\ref{example,2}, where the red dots are the images of the newly generated vertices from $ \mathcal{T}^{(n-1)}_{h,ur} $.}\label{fig5}
		%% https://en.wikibooks.org/wiki/LaTeX/Importing_Graphics#Importing_external_graphics
	\end{figure}
	Obviously, for $ p=2 $, the parameter domain $ \hat{\Omega} $ in Example~\ref{example,2} admits a partition $ \mathcal{T}_h:=\{K_{1,1}\} $ with $ K_{1,1}=\hat{\Omega} $. It is considered that the mesh $ \mathcal{T}_h $  is uniformly refined, and the $ n $-th uniform refinement is denoted by $ \mathcal{T}^{(n)}_{h,ur} $. The $ n $-th refinement of the physical domain $ \Omega $ is the image of the refined mesh $ \mathcal{T}^{(n)}_{h,ur} $ through the geometric mapping $ \boldsymbol{F} $. See Fig.~\ref{fig5} for two refinements of the physical domain. As demonstrated in the end of Subsubsection~\ref{subsec3.1}, the coefficients in the mapping form \eqref{eq:mapping, B,*} associated with the newly generated vertices  in each refinement are the only elements that undergo updating. 
	
	We use the approximation solution $ u_h(x,y) $ in $ \mathcal{V}_h^*\subset H^1(\Omega) $, which takes the form 
	\begin{equation*}
		u_h(x,y)=\sum_{j_1=1,j_2=1}^{2^{n},2^{n}}\sum_{i_1=1,i_2=1}^{2,2}u_{j_1,i_1;j_2,i_2}(\boldsymbol{B}^*_{j_1,i_1;j_2,i_2}\circ\boldsymbol{F}^{-1})(x,y).
	\end{equation*} 
	By the Hermite interpolation property, the boundary condition on the physical domain $ \Omega $ can be used to determine certain unknows $ u_{j_1,i_1;j_2,i_2} $ associated with the boundary points on the parameter domain $ \hat{\Omega} $, e.g., see \eqref{eq: boundary imposition,1} and \eqref{eq: boundary imposition,2}. The $ L^2 $-norm error $ \|u-u_h\|_{0} $ is not only taken into consideration; the following maximum-norm errors for the mesh vertices are also evaluated.
	\begin{equation*}
		\|u-u_h\|_{0,\infty,\mathcal{T}_{h,ur}^{(n)}}:= \max\{|u_{j_1,1;j_2,1}-u(\boldsymbol{F}(\xi_{j_1},\zeta_{j_2}) )|,\,\,\,j_1,j_2=1,2,3,...,2^n \},
	\end{equation*}
	and 
	\begin{align*}
		&|u-u_h|_{1,\infty,\mathcal{T}_{h,ur}^{(n)}}:= \max\bigg\{\bigg|u_{j_1,2;j_2,1}-\frac{\partial u(\boldsymbol{F}(\hat{x},\hat{y}) )}{\partial \hat{x}}\big|_{(\hat{x},\hat{y})=(\xi_{j_1},\zeta_{j_2})}\bigg|+\\
		&\bigg|u_{j_1,1;j_2,2}-\frac{\partial u(\boldsymbol{F}(\hat{x},\hat{y}) )}{\partial \hat{y}}\big|_{(\hat{x},\hat{y})=(\xi_{j_1},\zeta_{j_2})}\bigg|,\,\,\,j_1,j_2=1,2,3,...,2^n \bigg\},
	\end{align*}
	The error results under consideration are displayed in Table~\ref{table1}. The global $ L^2 $-norm errors show the optimal approximation power. Furthermore, the maximum-norm error $ \|u-u_h\|_{0,\infty,\mathcal{T}_{h,ur}^{(n)}} $ show the superconvergence property of function values for $ p = 2 $, which is consistent with the conclusion of the literature, e.g., \cite{Cosmin2015,Kumar2017}. Although control points in the physical domain may exhibit irregular distributions (see Fig.~\ref{fig2.1}), superconvergence still occurs as long as the parameter domain remains uniformly partitioned. While rigorous proof remains elusive, numerical results in isogeometric analysis validate this behavior, revealing a striking contrast with finite element analysis where uniform physical meshes are usually required for superconvergence. This demonstrates the superior practical value of the isogeometric framework in high-precision computational theories.
	\begin{table}[htbp]%% placement specifier
		%% Use tabular environment to tag the tabular data.
		%% https://en.wikibooks.org/wiki/LaTeX/Tables#The_tabular_environment
		\centering%% For centre alignment of tabular.
		\begin{tabular}{c c c c c c c}%% Table column specifiers
			%% Tabular cells are separated by &
			\hline
			n & $  \|u-u_h\|_{0} $ & order & $ \|u-u_h\|_{0,\infty,\mathcal{T}_{h,ur}^{(n)}} $ & order &$ |u-u_h|_{1,\infty,\mathcal{T}_{h,ur}^{(n)}} $& order \\
			\hline
			3 & 1.7136e-03 & $\setminus$ & 2.1246e-03  & $\setminus$ & 2.6996e-01&$\setminus$ \\ %% A tabular row ends with \\
			4 & 1.2227e-04 &3.8089 &2.5645e-04  & 3.0504 & 8.0462e-02&1.7464\\
			5 & 1.1578e-05 &3.4005 &1.9775e-05  & 3.6970 & 2.2329e-02 &1.8494\\
			6 & 1.3027e-06 &3.1518 &1.4327e-06  & 3.7869 & 5.6319e-03&1.9872\\
			7 & 1.5689e-07 &3.0537 &9.7855e-08  & 3.8719 & 1.4095e-03&1.9984\\
			\hline
		\end{tabular}
		%% Use \caption command for table caption and label.
		\caption{Error results for Example~\ref{example,2}.}\label{table1}
	\end{table}
	\subsubsection{2D biharmonic equation}
	\label{subsec4.2}
	\begin{figure}[htbp]%% placement specifier
		%% Use \includegraphics command to insert graphic files. Place graphics files in 
		%% working directory.
		\centering%% For centre alignment of image.
		\subfigure[Control net and physical domain $ \Omega $]{
			\begin{minipage}{.45\textwidth}
				\centering
				\includegraphics[width=180pt]{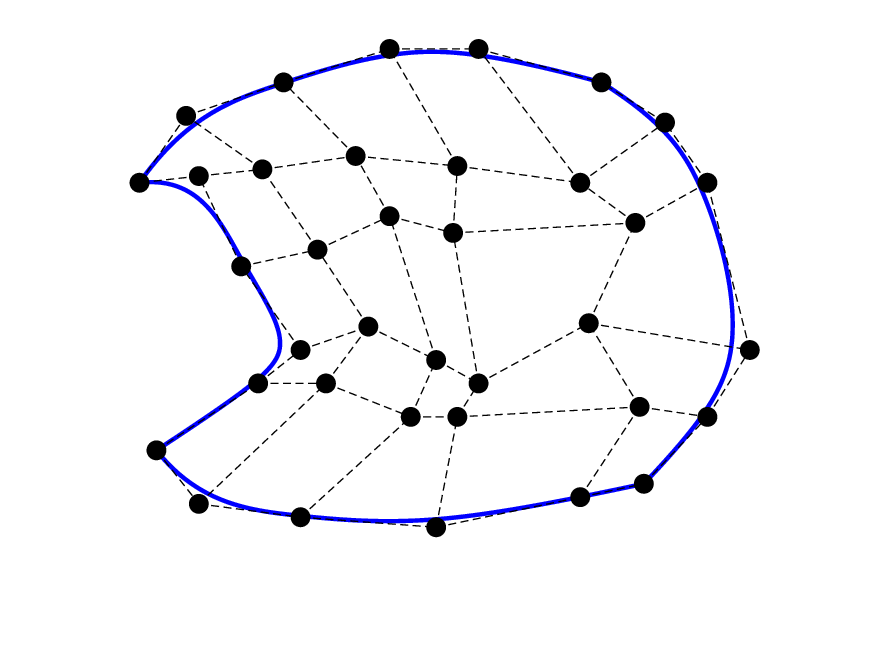}
				\label{fig4.1}
			\end{minipage}
		}
		\subfigure[Exact solution]{
			\begin{minipage}{.45\textwidth}
				\centering
				\includegraphics[width=180pt]{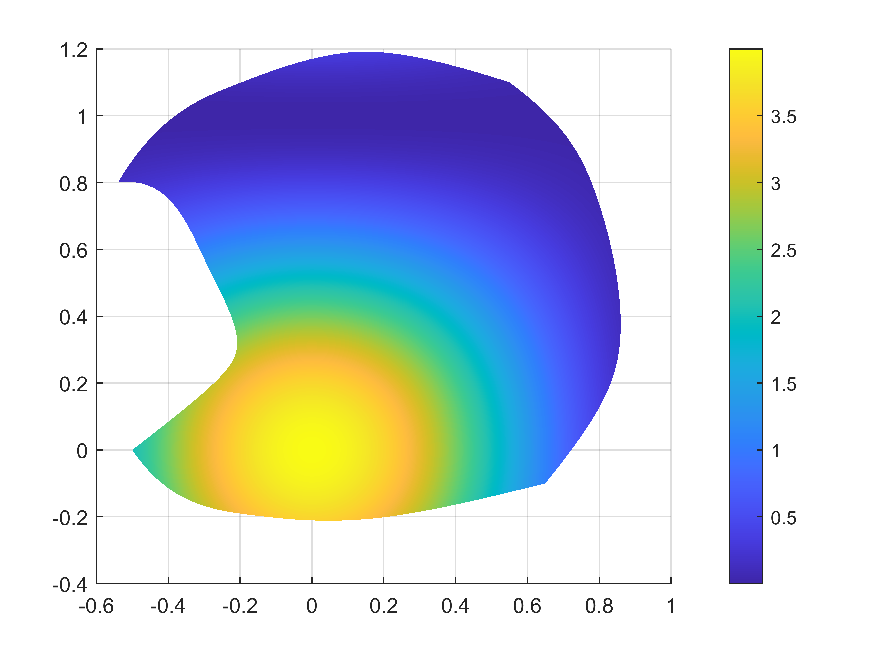}
				\label{fig4.2}
			\end{minipage}
		}
		%% Use \caption command for figure caption and label.
		\caption{The physical domain $ \Omega $ and exact solution  of Example~\ref{example,4}.}\label{fig4}
		%% https://en.wikibooks.org/wiki/LaTeX/Importing_Graphics#Importing_external_graphics
	\end{figure}
	\begin{table}[htbp]%% placement specifier
		%% Use tabular environment to tag the tabular data.
		%% https://en.wikibooks.org/wiki/LaTeX/Tables#The_tabular_environment
		\centering%% For centre alignment of tabular.
		{\scriptsize  \begin{tabular}{c c c}%% Table column specifiers
				%% Tabular cells are separated by &
				\hline
				$i_1,i_2 $ & $ \boldsymbol{C}_{1,i_1;1,i_2} $ &$ \boldsymbol{U}_{1,i_1;1,i_2} $ \\
				\hline
				1,1 & (-0.50, 0.00) & (-0.50, 0.00)  \\ %% A tabular row ends with \\
				2,1 & (-0.26, 0.20) & ( 0.90,-1.44)\\
				3,1 & (-0.16, 0.30) & ( 1.08, 7.56)\\
				1,2 & (-0.40,-0.16) & ( 2.16, 1.80)\\
				2,2 & (-0.10, 0.20) & ( 4.86, 13.0)\\ 
				3,2 & ( 0.00, 0.37) & (-38.9,-92.3)\\
				1,3 & (-0.16,-0.20) & (-10.3,-8.1)\\
				2,3 & ( 0.10, 0.10) & (-29.2,-60.8)\\
				3,3 & ( 0.16, 0.27) & ( 204, 452)\\ 
				\hline
			\end{tabular}\,\,\,
			\begin{tabular}{c c c}%% Table column specifiers
				%% Tabular cells are separated by &
				\hline
				$i_1,i_2 $ & $ \boldsymbol{C}_{2,i_1;1,i_2} $ &$ \boldsymbol{U}_{2,i_1;1,i_2} $ \\
				\hline
				1,1 & (-0.30,0.55) & (-0.54, 0.80)  \\ %% A tabular row ends with \\
				2,1 & (-0.40,0.82) & ( 0.99, 1.80)\\
				3,1 & (-0.54,0.80) & ( 0.27,-8.10)\\
				1,2 & (-0.12,0.60) & (-1.26,-0.18)\\
				2,2 & (-0.25,0.84) & (-3.24, 14.6)\\ 
				3,2 & (-0.43,1.00) & ( 21.9,-72.9)\\
				1,3 & ( 0.05,0.70) & (-4.86,-8.37)\\
				2,3 & (-0.03,0.88) & (-12.2, 94.8)\\
				3,3 & (-0.20,1.10) & ( 51.0, 437)\\ 
				\hline
			\end{tabular}\\[2mm]
			\begin{tabular}{c c c}%% Table column specifiers
				%% Tabular cells are separated by &
				\hline
				$i_1,i_2 $ & $ \boldsymbol{C}_{1,i_1;2,i_2} $ &$ \boldsymbol{U}_{1,i_1;2,i_2} $ \\
				\hline
				1,1 & (0.16,-0.23) & ( 0.65,-0.10)  \\ %% A tabular row ends with \\
				2,1 & (0.21, 0.10) & ( 1.35, 0.36)\\
				3,1 & (0.26, 0.20) & (-1.08,-0.27)\\
				1,2 & (0.50,-0.14) & ( 1.35, 1.80)\\
				2,2 & (0.64, 0.13) & ( 0.81,-5.67)\\ 
				3,2 & (0.52, 0.38) & (-1.70,-1.94)\\
				1,3 & (0.65,-0.10) & (-5.40,-5.40)\\
				2,3 & (0.80, 0.10) & ( 48.6, 21.9)\\
				3,3 & (0.90, 0.30) & ( 547, -6.56)\\ 
				\hline
			\end{tabular}\,\,\,
			\begin{tabular}{c c c}%% Table column specifiers
				%% Tabular cells are separated by &
				\hline
				$i_1,i_2 $ & $ \boldsymbol{C}_{2,i_1;2,i_2} $ &$ \boldsymbol{U}_{2,i_1;2,i_2} $ \\
				\hline
				1,1 & (0.20,0.65) & ( 0.55, 1.10)  \\ %% A tabular row ends with \\
				2,1 & (0.21,0.85) & ( 2.61,-0.90)\\
				3,1 & (0.05,1.20) & ( 9.99,-5.40)\\
				1,2 & (0.63,0.68) & (-1.35, 1.08)\\
				2,2 & (0.50,0.80) & ( 7.29,-22.7)\\ 
				3,2 & (0.26,1.20) & ( 63.2,-148)\\
				1,3 & (0.80,0.80) & (-5.40, 1.62)\\
				2,3 & (0.70,0.98) & ( 36.5, -151)\\
				3,3 & (0.55,1.10) & ( 233,-1035)\\ 
				\hline
		\end{tabular}}
		%% Use \caption command for table caption and label.
		\caption{The control points $ \boldsymbol{C}_{j_1,i_1;j_2,i_2} $ for the control net in Fig.~\ref{fig4}, and the related interpolatory points $ \boldsymbol{U}_{j_1,i_1;j_2,i_2} $ derived from the Hermite interpolation property.}
		\label{table,example,4}
	\end{table}
	\begin{example}\label{example,4}
		In this example, we consider the parameter domain $ \hat{\Omega}:=\langle\Xi_{\xi}\rangle\otimes\langle\Xi_{\zeta}\rangle  $ in the $ \hat{x}\hat{y} $-plane, where
		\begin{equation*}
			\Xi_{\xi} =\Xi_{\zeta}=\{ 0,0,0,0,\frac{1}{3},\frac{2}{3},1,1,1,1\}.
		\end{equation*}
		The physical domain $ \Omega$ is obtained through the geometric mapping
		\begin{equation*}
			\boldsymbol{F}(\hat{x},\hat{y}) = \sum_{j_1=1,j_2=1}^{2,2}\sum_{i_1=1,i_2=1}^{3,3}\boldsymbol{C}_{j_1,i_1;j_2,i_2}\boldsymbol{B}_{j_1,i_1;j_2,i_2}(\hat{x},\hat{y}),
		\end{equation*}
		where the control points $ \boldsymbol{C}_{j_1,i_1;j_2,i_2} $ are given in Table~\ref{table,example,4}. The control net and the physical domain are illustrated in Fig.~\ref{fig2.1}. Moreover, using the new basis \eqref{def:new basis,2D} for the mapping \eqref{eq:mapping, B} gives 
		\begin{equation*}
			\boldsymbol{F}(\hat{x},\hat{y}) = \sum_{j_1=1,j_2=1}^{2,2}\sum_{i_1=1,i_2=1}^{3,3}\boldsymbol{U}_{j_1,i_1;j_2,i_2}\boldsymbol{B}^{*}_{j_1,i_1;j_2,i_2}(\hat{x},\hat{y}),
		\end{equation*}
		where  the points $ \boldsymbol{U}_{j_1,i_1;j_2,i_2} $ with interpolatory properties in the form \eqref{eq:mapping, B,*} are given in Table~\ref{table,example,4}.	We solve the 2D biharmonic equation 
		\begin{eqnarray*}\label{4th-order PDE}
			\left\{
			\begin{aligned}
				&\Delta^2 u +u =f\quad \mbox{in $ \Omega $},\\
				&u=g_1 \quad \mbox{on $ \partial\Omega $},\\
				&\nabla u\cdot \textbf{n} =g_2\quad \mbox{on $ \partial\Omega $},
			\end{aligned}
			\right.
		\end{eqnarray*}
		with the following exact solution:
		\begin{equation*}
			u(x,y) = (1+\cos(\pi x))(1+\cos{(\pi y)}),
		\end{equation*}
		where the functions $ f $, $ g_1 $ and $ g_2 $ are given by the exact solution. The image of the exact solution is presented in Fig.~\ref{fig4.2}. 
	\end{example}
	\begin{figure}[htbp]%% placement specifier
		%% Use \includegraphics command to insert graphic files. Place graphics files in 
		%% working directory.
		\centering%% For centre alignment of image.
		\subfigure[$ \Omega:=\boldsymbol{F}(\hat{\Omega}) $]{
			\begin{minipage}{.3\textwidth}
				\centering
				\includegraphics[width=130pt]{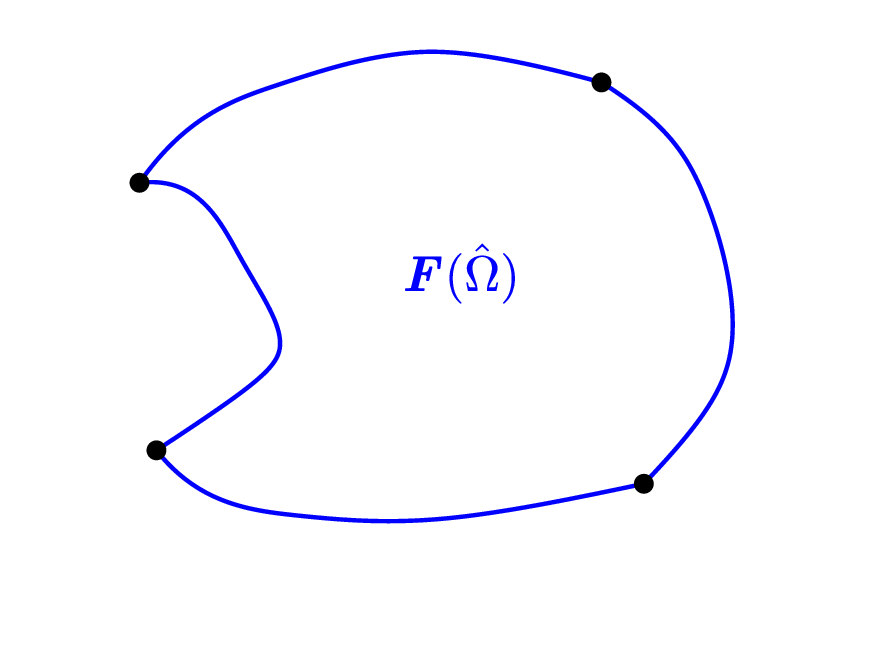}
				\label{fig6.1}
			\end{minipage}
		}
		\subfigure[$ \boldsymbol{F}(\mathcal{T}^{(1)}_{h,ur}) $]{
			\begin{minipage}{.3\textwidth}
				\centering
				\includegraphics[width=130pt]{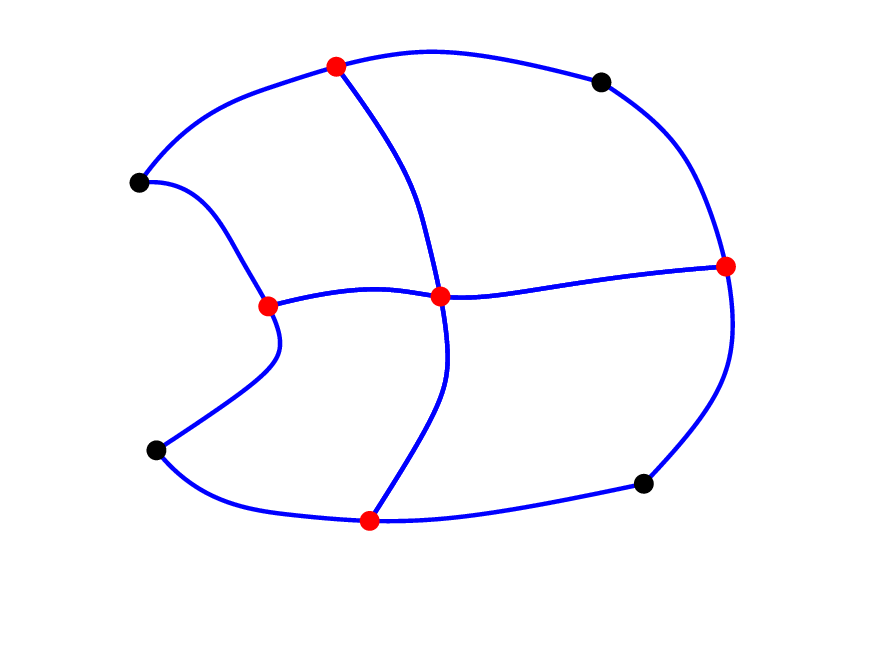}
				\label{fig6.2}
			\end{minipage}
		}
		\subfigure[$ \boldsymbol{F}(\mathcal{T}^{(2)}_{h,ur}) $]{
			\begin{minipage}{.3\textwidth}
				\centering
				\includegraphics[width=130pt]{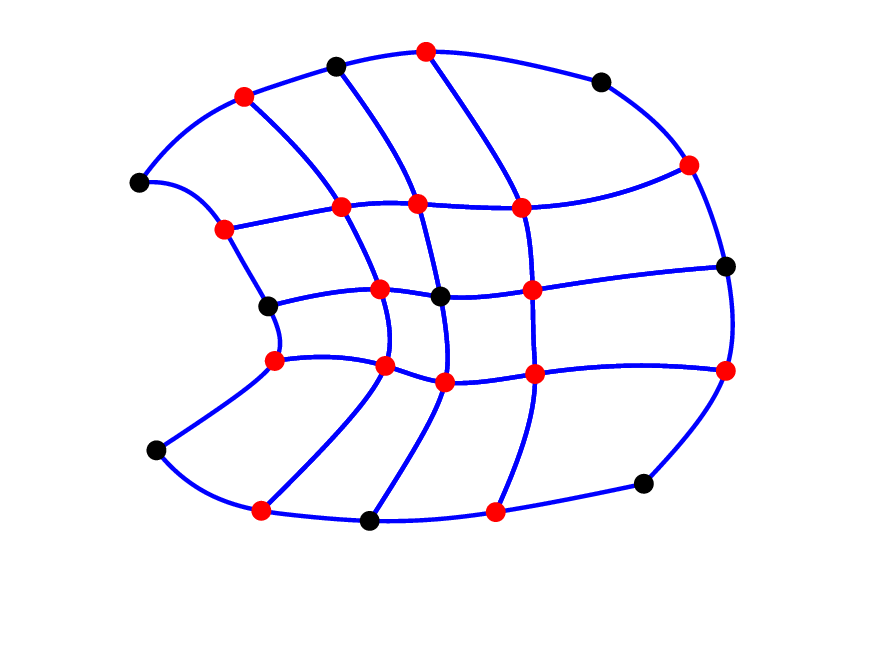}
				\label{fig6.3}
			\end{minipage}
		}
		%% Use \caption command for figure caption and label.
		\caption{The images of the meshes $ \mathcal{T}^{(n)}_{h,ur} $ through the geometric mapping $ \boldsymbol{F} $ in Example~\ref{example,4}, where the red dots are the images of the newly generated vertices from $ \mathcal{T}^{(n-1)}_{h,ur} $.}\label{fig6}
		%% https://en.wikibooks.org/wiki/LaTeX/Importing_Graphics#Importing_external_graphics
	\end{figure}
	Obviously, for $ p=3 $, the parameter domain $ \hat{\Omega} $ in Example~\ref{example,4} admits a partition $ \mathcal{T}_h:=\{K_{1,1}\} $ with $ K_{1,1}=\hat{\Omega} $. Consider the same notations as Subsubsection~\ref{subsec4.1}. Then, we use the approximation solution $ u_h(x,y) $ taking the  form 
	\begin{equation*}
		u_h(x,y)=\sum_{j_1=1,j_2=1}^{2^{n},2^{n}}\sum_{i_1=1,i_2=1}^{3,3}u_{j_1,i_1;j_2,i_2}(\boldsymbol{B}^*_{j_1,i_1;j_2,i_2}\circ\boldsymbol{F}^{-1})(x,y).
	\end{equation*} 
	The following two maximum-norm errors for the mesh vertices are also in consideration.
	\begin{align*}
		&|u-u_h|_{2,\infty,\mathcal{T}_{h,ur}^{(n)}}:= \max\bigg\{\bigg|u_{j_1,3;j_2,1}-\frac{\partial^2 u(\boldsymbol{F}(\hat{x},\hat{y}) )}{\partial \hat{x}^2}\big|_{(\hat{x},\hat{y})=(\xi_{j_1},\zeta_{j_2})}\bigg|+\\
		&\bigg|u_{j_1,1;j_2,3}-\frac{\partial^2 u(\boldsymbol{F}(\hat{x},\hat{y}) )}{\partial \hat{y}^2}\big|_{(\hat{x},\hat{y})=(\xi_{j_1},\zeta_{j_2})}\bigg|,\,\,\,j_1,j_2=1,2,3,...,2^n \bigg\},
	\end{align*}
	and
	\begin{align*}
		&|u-u_h|_{1,\infty,\mathcal{T}_{h,ur}^{(n)}\cap \hat{\Omega}_{in}}:= \max\bigg\{\bigg|u_{j_1,2;j_2,1}-\frac{\partial u(\boldsymbol{F}(\hat{x},\hat{y}) )}{\partial \hat{x}}\big|_{(\hat{x},\hat{y})=(\xi_{j_1},\zeta_{j_2})}\bigg|+\\
		&\bigg|u_{j_1,1;j_2,2}-\frac{\partial u(\boldsymbol{F}(\hat{x},\hat{y}) )}{\partial \hat{y}}\big|_{(\hat{x},\hat{y})=(\xi_{j_1},\zeta_{j_2})}\bigg|,\,\mbox{$ j_1,j_2=1,2,3,...,2^n $, s.t., $ (\xi_{j_1},\zeta_{j_2})\in\hat{\Omega}_{in} $} \bigg\},
	\end{align*}
	where $ \hat{\Omega}_{in}:=[0.1,0.9]^2 $.
	The error results under consideration are displayed in Table~\ref{table2}. The global $ L^2 $-norm errors show the optimal approximation power. Furthermore, the maximum-norm error $ |u-u_h|_{1,\infty,\mathcal{T}_{h,ur}^{(n)}\cap \hat{\Omega}_{in}} $ in the inner domain $ \boldsymbol{F}(\hat{\Omega}_{in}) $ show the  superconvergence property of the first derivative for $ p = 3 $. This is also consistent with the conclusion of Subsubsection~\ref{subsec4.1} and previous studies \cite{Cosmin2015,Kumar2017}.
	\begin{table}[htbp]%% placement specifier
		%% Use tabular environment to tag the tabular data.
		%% https://en.wikibooks.org/wiki/LaTeX/Tables#The_tabular_environment
		\centering%% For centre alignment of tabular.
		\begin{tabular}{c c c c c c c}%% Table column specifiers
			%% Tabular cells are separated by &
			\hline
			n & $  \|u-u_h\|_{0} $ & order & $ \|u-u_h\|_{0,\infty,\mathcal{T}_{h,ur}^{(n)}} $ & order &$ |u-u_h|_{1,\infty,\mathcal{T}_{h,ur}^{(n)}} $& order \\
			\hline
			3 &1.7239e-04& $\setminus$ &    2.2023e-04  & $\setminus$ &    4.5743e-03&$\setminus$ \\ %% A tabular row ends with \\
			4 &9.8334e-06& 4.1318&   2.5762e-05& 3.0957&    8.3218e-04&   2.4586\\
			5 &5.7779e-07& 4.0891&   2.2805e-06 & 3.4979&    1.0658e-04&   2.9650\\
			6 &3.5150e-08& 4.0389 & 1.9161e-07 & 3.5731 &    1.4148e-05&   2.9132\\
			\hline
		\end{tabular}\\[2mm]
		\begin{tabular}{ c c c c}%% Table column specifiers
			%% Tabular cells are separated by &
			\hline
			$ |u-u_h|_{2,\infty,\mathcal{T}_{h,ur}^{(n)}} $ & order &$ |u-u_h|_{1,\infty,\mathcal{T}_{h,ur}^{(n)}\cap \hat{\Omega}_{in} } $& order \\
			\hline 
			1.0410e+01  & $\setminus$ &    4.5743e-03 &$\setminus$ \\ %% A tabular row ends with \\
			3.1491e+00 &    1.7250&    3.2349e-04&  3.8218\\
			8.0843e-01&    1.9617&   2.0052e-05 &   4.0119\\
			2.0355e-01&    1.9898&    1.4754e-06 &  3.7646\\
			\hline
		\end{tabular}
		%% Use \caption command for table caption and label.
		\caption{Error results for Example~\ref{example,4}.}\label{table2}
	\end{table}
\subsection{Numerical examples in three dimension}\label{subsec5.2}
This subsection investigates the computational efficiency of element-based B-spline spaces in 3D isogeometric analysis under non-uniform parametric grids. By leveraging the element-wise representations of both the geometric mapping  $ \boldsymbol{F} $ and basis functions, our implementation achieves a unified assembly pattern across all elements, mirroring the procedural efficiency of finite element methods. Numerical benchmarks are performed using an in-house MATLAB framework following the proposed methodology, with comparative results obtained from the MATLAB GeoPDEs library \cite{GeoPDEs} under identical hardware configurations (11th Gen Intel Core i5-11300H, 16 GB RAM).

The proposed isogeometric workflow based on element-based B-spline spaces introduces only one extra step compared to standard isoparametric finite element analysis: converting CAD control points into interpolatory points via Hermite interpolation. This preprocessing step adds minimal computational cost to the overall simulation. 

We test the method on two 3D CAD models: a duck and a human face, see Fig.~\ref{fig: 3D models}. Both 3D models share an identical parameter domain $ \hat{\Omega}:=\langle\Xi_{\xi}\rangle\otimes\langle\Xi_{\zeta}\rangle\otimes\langle\Xi_{\eta}\rangle  $ in the $ \hat{x}\hat{y}\hat{z}$-space with the non-uniform configuration:
\begin{align*}
\Xi_{\xi} = \{&0,0,0,0,0.04,0.08,0.12,0.18,0.24,0.3,0.3233,0.3467,0.37,0.4133,0.4567,0.5,\\
	   &0.5333,0.5667,0.6,0.65,0.7,0.75,0.8,0.85,0.9,0.9333,0.9667,1,1,1,1\},\\[1.5mm]
	\Xi_{\zeta} = \{&0,0,0,0,0.04,0.08,0.12,0.1467,0.1733,0.2,0.2567,0.3133, 0.37,0.4133,0.4567,0.5,\\
	&0.5167,0.5333,0.55,0.6167,0.6833,0.75,0.7833,0.8167,0.85,0.9,0.95,1,1 ,1,1\},\\[1.5mm]
	\Xi_{\eta} = \{&0,0,0,0,0.0333,0.0667,0.1,0.15,0.2,0.25,0.2667,0.2833, 0.3,0.3667,0.4333,0.5,\\
	&0.54,0.58,0.62,0.6633,0.7067,0.75,0.79,0.83,0.87,0.9133,0.9567,1,1,1, 1 \}.
\end{align*}
The physical domain $ \Omega_{duck}$ and $ \Omega_{face}$ are obtained through the geometric mappings
\begin{equation*}
	\boldsymbol{F}_{duck}(\hat{x},\hat{y},\hat{z}) = \sum_{j_1=1,j_2=1,j_3=1}^{9,9,9}\sum_{i_1=1,i_2=1,i_3=1}^{3,3,3}\boldsymbol{C}^{duck}_{j_1,i_1;j_2,i_2;j_3,i_3}\boldsymbol{B}_{j_1,i_1;j_2,i_2;j_3,i_3}(\hat{x},\hat{y},\hat{z}),
\end{equation*}
and 
\begin{equation*}
	\boldsymbol{F}_{face}(\hat{x},\hat{y},\hat{z}) = \sum_{j_1=1,j_2=1,j_3=1}^{9,9,9}\sum_{i_1=1,i_2=1,i_3=1}^{3,3,3}\boldsymbol{C}^{face}_{j_1,i_1;j_2,i_2;j_3,i_3}\boldsymbol{B}_{j_1,i_1;j_2,i_2;j_3,i_3}(\hat{x},\hat{y},\hat{z}),
\end{equation*}
respectively, where $ \boldsymbol{C}^{duck}_{j_1,i_1;j_2,i_2;j_3,i_3}$ and  $ \boldsymbol{C}^{face}_{j_1,i_1;j_2,i_2;j_3,i_3}$ are the control points  of the 3D CAD models.  For the above geometric mapping, using the new basis \eqref{def:new basis,2D} in the 3D settings gives 
\begin{equation*}
	\boldsymbol{F}_{duck}(\hat{x},\hat{y},\hat{z}) = \sum_{j_1=1,j_2=1,j_3=1}^{9,9,9}\sum_{i_1=1,i_2=1,i_3=1}^{3,3,3}\boldsymbol{U}^{duck}_{j_1,i_1;j_2,i_2;j_3,i_3}\boldsymbol{B}^*_{j_1,i_1;j_2,i_2;j_3,i_3}(\hat{x},\hat{y},\hat{z}),
\end{equation*}
and 
\begin{equation*}
	\boldsymbol{F}_{face}(\hat{x},\hat{y},\hat{z}) = \sum_{j_1=1,j_2=1,j_3=1}^{9,9,9}\sum_{i_1=1,i_2=1,i_3=1}^{3,3,3}\boldsymbol{U}^{face}_{j_1,i_1;j_2,i_2;j_3,i_3}\boldsymbol{B}^*_{j_1,i_1;j_2,i_2;j_3,i_3}(\hat{x},\hat{y},\hat{z}),
\end{equation*}
where $ \boldsymbol{U}^{duck}_{j_1,i_1;j_2,i_2;j_3,i_3}$ and  $ \boldsymbol{U}^{face}_{j_1,i_1;j_2,i_2;j_3,i_3}$ are the interpolatory  points  of the 3D CAD models. Then, each of the global mappings admits an element-wise decomposition following the formulation in the equation \eqref{eq:mapping, B,*,element}, which serves as its two-dimensional counterpart.
	\begin{figure}[htbp]%% placement specifier
	%% Use \includegraphics command to insert graphic files. Place graphics files in 
	%% working directory.
	\centering%% For centre alignment of image.
	\subfigure[duck]{
		\begin{minipage}{.45\textwidth}
			\centering
			\includegraphics[width=100pt]{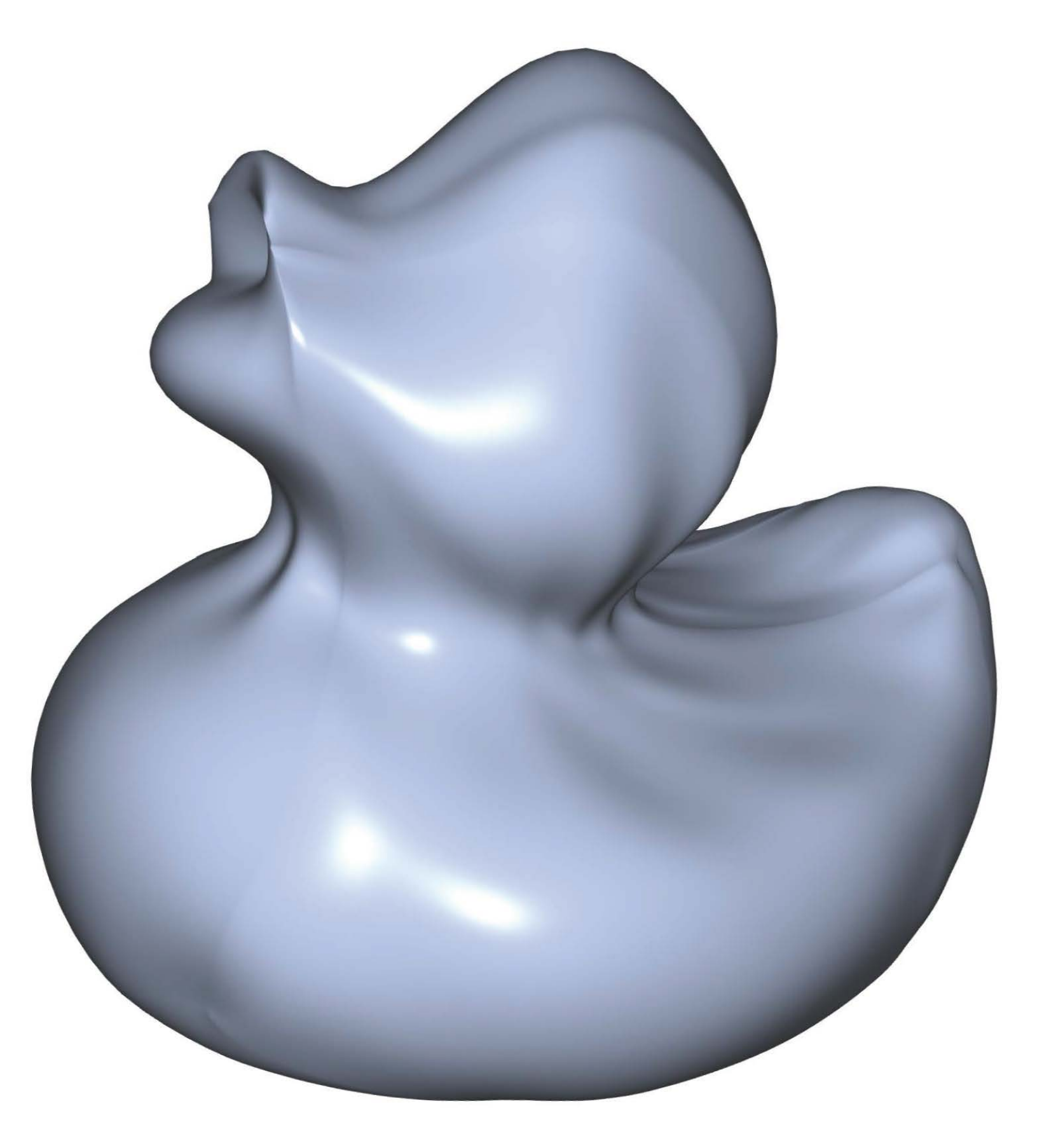}
%			\label{fig4.1}
		\end{minipage}
	}
	\subfigure[human face]{
	\begin{minipage}{.45\textwidth}
		\centering
		\includegraphics[width=75pt]{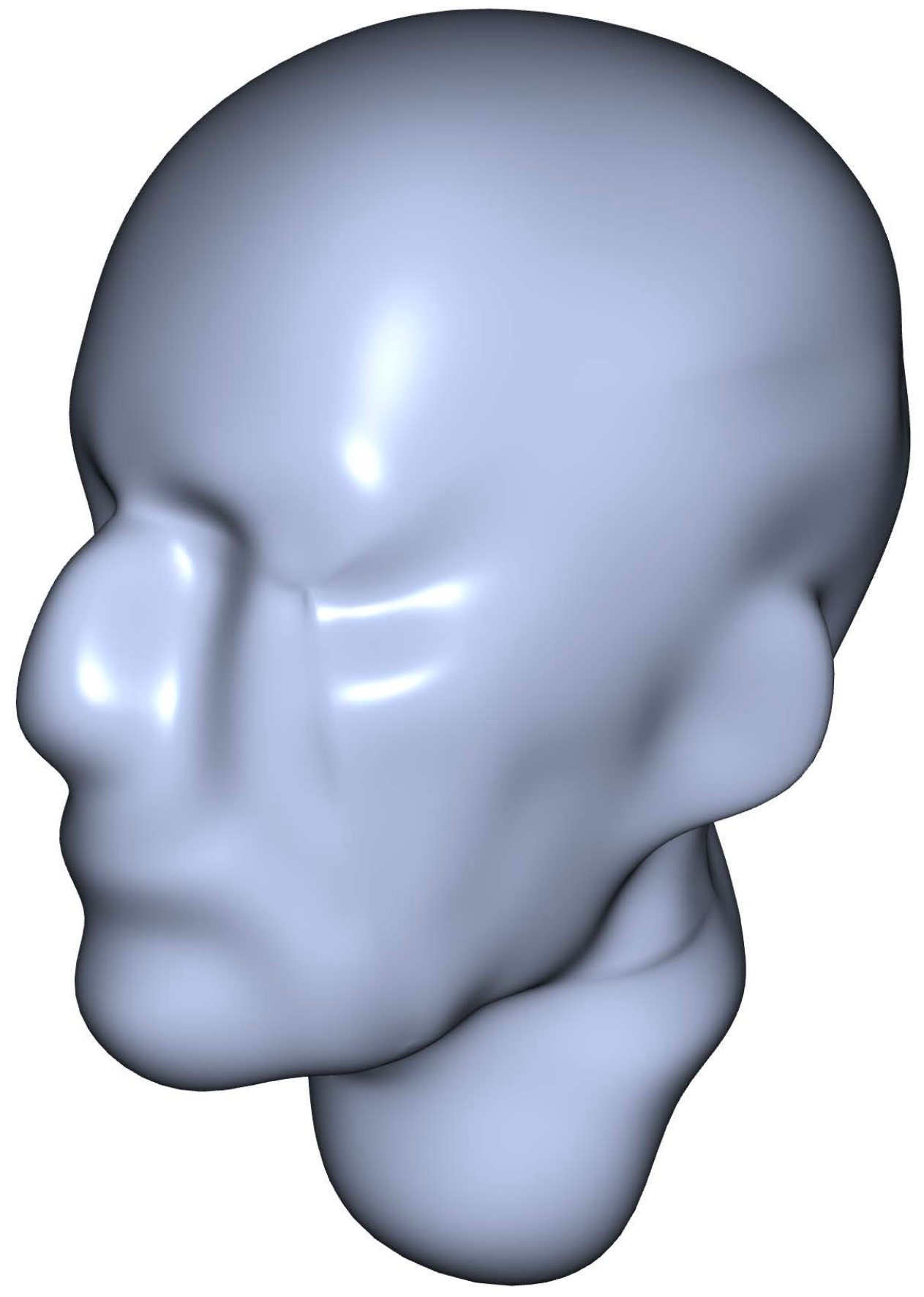}
		%			\label{fig4.1}
	\end{minipage}
}
%	\subfigure[Exact solution]{
%		\begin{minipage}{.45\textwidth}
%			\centering
%			\includegraphics[width=180pt]{figures/ExactSolutionBi.eps}
%			\label{fig4.2}
%		\end{minipage}
%	}
	%% Use \caption command for figure caption and label.
	\caption{Graphics of 3D CAD models.}\label{fig: 3D models}
\end{figure}
\begin{example}\label{example,5}
	We solve the 3D Possion equation 
	\begin{eqnarray*}
		\left\{
		\begin{aligned}
			&-\Delta u =f\quad \mbox{in $\Omega_{duck}(\Omega_{face})$},\\
			&u=g \quad \mbox{on $\partial\Omega_{duck}(\partial\Omega_{face})$},
		\end{aligned}
		\right.
	\end{eqnarray*}  
	with the following exact solution:
	\begin{equation*}
		u(x,y) = \sin(\pi x)\sin{(\pi y)}\sin{(\pi z)},
	\end{equation*}
	where the functions $ f $ and $ g $ are given by the exact solution. 
\end{example}
	\begin{example}\label{example,6}
  We solve the 3D biharmonic equation 
	\begin{eqnarray*}\label{4th-order PDE,3D}
		\left\{
		\begin{aligned}
			&\Delta^2 u =f\quad \mbox{in $\Omega_{duck}(\Omega_{face})$},\\
			&u=g_1 \quad \mbox{on $\partial\Omega_{duck}(\partial\Omega_{face}) $},\\
			&\nabla u\cdot \textbf{n} =g_2\quad \mbox{on $ \partial\Omega_{duck}(\partial\Omega_{face}) $},
		\end{aligned}
		\right.
	\end{eqnarray*}
	with the following exact solution:
	\begin{equation*}
		u(x,y) = \sin(\pi x)\sin{(\pi y)}\sin{(\pi z)},
	\end{equation*}
	where the functions $ f $, $ g_1 $ and $ g_2 $ are given by the exact solution. 
\end{example}
\begin{figure}[htbp]%% placement specifier
	%% Use \includegraphics command to insert graphic files. Place graphics files in 
	%% working directory.
	\centering%% For centre alignment of image.
	\subfigure[Proposed element-based B-spline workflow]{
		\begin{minipage}{.45\textwidth}
			\centering
			\includegraphics[width=190pt]{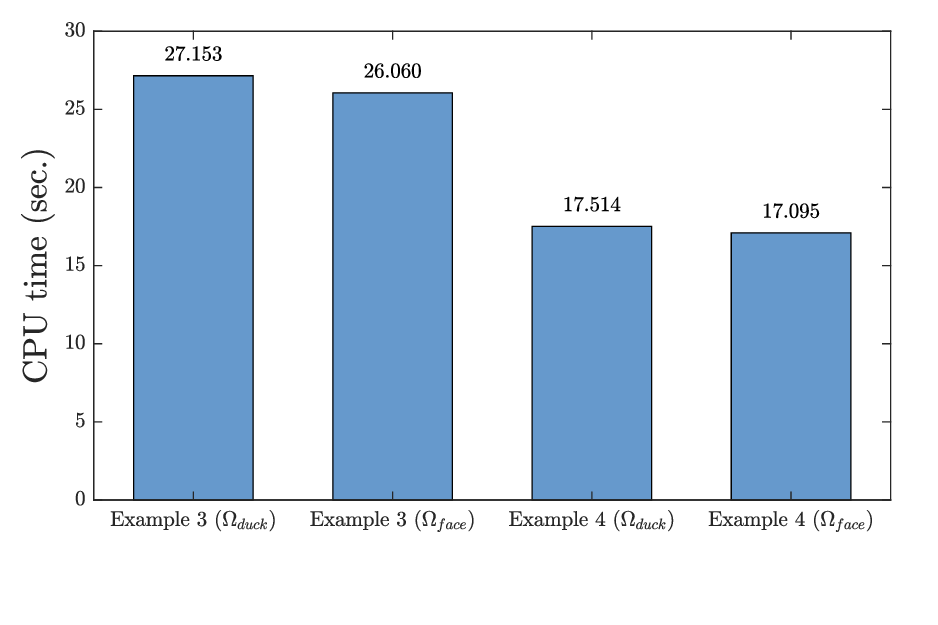}
			%			\label{fig4.1}
		\end{minipage}
	}
	\subfigure[Conventional IgA (GeoPDEs \cite{GeoPDEs})]{
		\begin{minipage}{.45\textwidth}
			\centering
			\includegraphics[width=190pt]{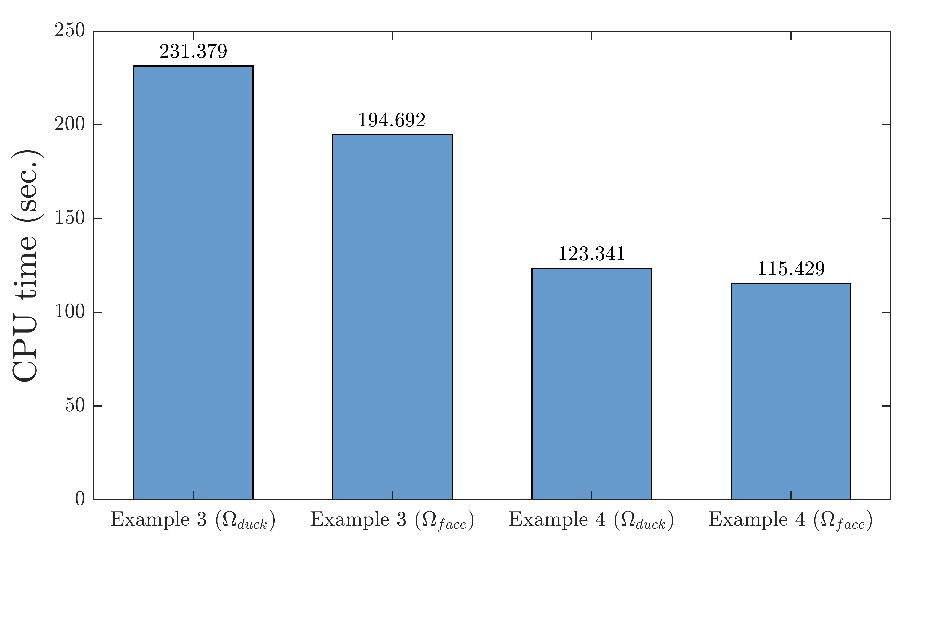}
			%			\label{fig4.1}
		\end{minipage}
	}
	\caption{Computational efficiency comparison of matrix assembly in IgA.}\label{fig: CPU time}
\end{figure}
Example~\ref{example,5} and \ref{example,6} benchmark the proposed isogeometric workflow based on element-based B-spline spaces against the MATLAB GeoPDEs library \cite{GeoPDEs}, with CPU time comparisons detailed in Fig.~\ref{fig: CPU time}. 

The observed reduction in CPU time for the 3D biharmonic equation compared to the 3D Poisson equation arises from fundamental differences in stiffness matrix construction. For the biharmonic problem, stiffness matrices involve scalar products of basis functions over the domain, whereas Poisson matrices require vector-valued inner products, inherently demanding higher computational intensity. 

On the other hand, as evidenced by the two subfigures of Fig.~\ref{fig: CPU time}, the conventional isogeometric method exhibits 6.5–8.5 times higher computational costs than the proposed element-based B-spline formulation. This efficiency gain arises from the element-based representation of both basis functions and geometric mappings, which enables identical local basis function expressions across all elements. Unlike conventional methods requiring numerical quadrature points for each individual basis function, particularly in parameter domains with non-uniform knots, our formulation allows precomputing quadrature points on a single element and reusing them throughout the domain. This optimization stems from the observation that all basis function inner products follow a fixed pattern that can be efficiently tabulated. Crucially, this accelerated computation paradigm achieves FEM-like matrix assembly efficiency while preserving the exact geometry representation inherent to isogeometric analysis.
\begin{remark}
%The Hermite interpolation process converting CAD control points to interpolatory points demonstrates negligible computational overhead (0.04-0.06s across all test cases), constituting less than 0.1\% of the total computation time.
The preprocessing phase implementing Hermite interpolation for CAD control points conversion requires 0.04-0.06 seconds in all tested cases.
\end{remark}
	\section{Conclusion}
	\label{sec5}
	
	This paper presents a novel framework for constructing B-spline basis function spaces through element-wise local bases, unifying geometric precision in isogeometric analysis with the computational flexibility in finite element analysis. By defining local basis functions on B-spline elements and reconstructing global bases via recursive linear combinations, we unify the classical B-spline structure with element-wise computational workflows.
	The proposed Hermite interpolation schemes for arbitrary smoothness constraints achieve optimal approximation error estimates and allow direct boundary condition enforcement. Furthermore, the decomposition of CAD-generated global mappings into local B-spline element mappings enables efficient numerical integration, even in non-uniform knot configurations. 
	
	Looking ahead, the newly revealed  element-wise structure of B-spline basis function spaces not only facilitates enhanced computational efficiency but also provides a natural framework for rigorous mathematical analysis. Our future investigations will adopt advanced finite element methodologies to systematically quantify the computational efficiency and develop theoretical frameworks addressing the superconvergence phenomena.  Meanwhile, the application of the B-spline element framework to adaptive local refinement strategies (e.g., hierarchical B-splines) coupled with rigorous theoretical analyses constitutes another potential direction.

	%\begin{acknowledgements}
	%\end{acknowledgements}


	
	
	\bibliographystyle{spmpsci}        % basic style, author-year citations
	\bibliography{SErefer}

\begin{thebibliography}{10}
\providecommand{\url}[1]{{#1}}
\providecommand{\urlprefix}{URL }
\expandafter\ifx\csname urlstyle\endcsname\relax
  \providecommand{\doi}[1]{DOI~\discretionary{}{}{}#1}\else
  \providecommand{\doi}{DOI~\discretionary{}{}{}\begingroup
  \urlstyle{rm}\Url}\fi

\bibitem{Alfeld2009}
Alfeld, P., Sorokina, T.: Two tetrahedral {$C^1$} cubic macro elements.
\newblock J. Approx. Theory \textbf{157}(1), 53--69 (2009)

\bibitem{Cosmin2015}
Anitescu, C., Jia, Y., Zhang, Y.J., Rabczuk, T.: An isogeometric collocation
  method using superconvergent points.
\newblock Comput. Methods Appl. Mech. Engrg. \textbf{284}, 1073--1097 (2015)

\bibitem{Antolin2015}
Antolin, P., Buffa, A., Calabr\`o, F., Martinelli, M., Sangalli, G.: Efficient
  matrix computation for tensor-product isogeometric analysis: the use of sum
  factorization.
\newblock Comput. Methods Appl. Mech. Engrg. \textbf{285}, 817--828 (2015)

\bibitem{Bazilevs2006}
Bazilevs, Y., Beir\~ao~da Veiga, L., Cottrell, J.A., Hughes, T.J.R., Sangalli,
  G.: Isogeometric analysis: approximation, stability and error estimates for
  {$h$}-refined meshes.
\newblock Math. Models Methods Appl. Sci. \textbf{16}(7), 1031--1090 (2006)

\bibitem{Boor1978}
de~Boor, C.: A practical guide to splines.
\newblock Springer-Verlag, New York-Berlin (1978)

\bibitem{Borden2011}
Borden, M.J., Scott, M.A., Evans, J.A., Hughes, T.J.R.: Isogeometric finite
  element data structures based on {B}\'ezier extraction of {NURBS}.
\newblock Internat. J. Numer. Methods Engrg. \textbf{87}(1-5), 15--47 (2011)

\bibitem{Calabro2017}
Calabr\`o, F., Sangalli, G., Tani, M.: Fast formation of isogeometric
  {G}alerkin matrices by weighted quadrature.
\newblock Comput. Methods Appl. Mech. Engrg. \textbf{316}, 606--622 (2017)

\bibitem{Cohen2013}
Cohen, E., Lyche, T., Riesenfeld, R.F.: A {B}-spline-like basis for the
  {P}owell-{S}abin 12-split based on simplex splines.
\newblock Math. Comp. \textbf{82}(283), 1667--1707 (2013)

\bibitem{Hughes2009}
Cottrell, J., Hughes, T., Bazilevs, Y.: Isogeometric analysis.
\newblock John Wiley \& Sons, Ltd., Chichester (2009)

\bibitem{Ronald1994}
Graham, R.L., Knuth, D.E., Patashnik, O.: Concrete mathematics.
\newblock Addison-Wesley Publishing Company, Reading, MA (1994)

\bibitem{Guo2015}
Guo, Y., Ruess, M.: Nitsche's method for a coupling of isogeometric thin shells
  and blended shell structures.
\newblock Comput. Methods Appl. Mech. Engrg. \textbf{284}, 881--905 (2015)

\bibitem{Hollig2003}
H\"ollig, K.: Finite element methods with {B}-splines, \emph{Frontiers in
  Applied Mathematics}, vol.~26.
\newblock Society for Industrial and Applied Mathematics (SIAM), Philadelphia,
  PA (2003)

\bibitem{Hollig2001}
H\"ollig, K., Reif, U., Wipper, J.: Weighted extended {B}-spline approximation
  of {D}irichlet problems.
\newblock SIAM J. Numer. Anal. \textbf{39}(2), 442--462 (2001)

\bibitem{Hu2015}
Hu, J., Zhang, S.: The minimal conforming {$H^k$} finite element spaces on
  {$R^n$} rectangular grids.
\newblock Math. Comp. \textbf{84}(292), 563--579 (2015)

\bibitem{hughes2000finite}
Hughes, T.J.R.: The finite element method.
\newblock Linear Static and Dynamic Finite Element Analysis  (2000)

\bibitem{Hughes2005IGA}
Hughes, T.J.R., Cottrell, J.A., Bazilevs, Y.: Isogeometric analysis: {CAD},
  finite elements, {NURBS}, exact geometry and mesh refinement.
\newblock Comput. Methods Appl. Mech. Engrg. \textbf{194}(39-41), 4135--4195
  (2005)

\bibitem{Kumar2017}
Kumar, M., Kvamsdal, T., Johannessen, K.A.: Superconvergent patch recovery and
  a posteriori error estimation technique in adaptive isogeometric analysis.
\newblock Comput. Methods Appl. Mech. Engrg. \textbf{316}, 1086--1156 (2017)

\bibitem{Lai2007}
Lai, M., Schumaker, L.L.: Spline functions on triangulations, vol. 110.
\newblock Cambridge University Press, Cambridge (2007)

\bibitem{Lyche2022FCM}
Lyche, T., Manni, C., Speleers, H.: Construction of {$C^2$} cubic splines on
  arbitrary triangulations.
\newblock Found. Comput. Math. \textbf{22}(5), 1309--1350 (2022)

\bibitem{Lyche2022}
Lyche, T., Merrien, J.L., Sauer, T.: Simplex-splines on the {C}lough-{T}ocher
  split with arbitrary smoothness.
\newblock pp. 85--121. Springer, Cham (2022)

\bibitem{Pan2020}
Pan, M., J\"uttler, B., Giust, A.: Fast formation of isogeometric {G}alerkin
  matrices via integration by interpolation and look-up.
\newblock Comput. Methods Appl. Mech. Engrg. \textbf{366}, 113005, 25 (2020)

\bibitem{piegl2012nurbs}
Piegl, L., Tiller, W.: The NURBS book.
\newblock Springer Science \& Business Media (2012)

\bibitem{Sange2019}
Sande, E., Manni, C., Speleers, H.: Sharp error estimates for spline
  approximation: explicit constants, {$n$}-widths, and eigenfunction
  convergence.
\newblock Math. Models Methods Appl. Sci. \textbf{29}(6), 1175--1205 (2019)

\bibitem{Sande2020}
Sande, E., Manni, C., Speleers, H.: Explicit error estimates for spline
  approximation of arbitrary smoothness in isogeometric analysis.
\newblock Numer. Math. \textbf{144}(4), 889--929 (2020)

\bibitem{Schaeuble2017}
Schaeuble, A.K., Tkachuk, A., Bischoff, M.: Variationally consistent inertia
  templates for {B}-spline- and {NURBS}-based {FEM}: inertia scaling and
  customization.
\newblock Comput. Methods Appl. Mech. Engrg. \textbf{326}, 596--621 (2017)

\bibitem{Chung1979}
Shih, C.T.: On a spline finite element method.
\newblock Math. Numer. Sinica \textbf{1}(1), 50--72 (1979)

\bibitem{Speleers2015}
Speleers, H.: A new {B}-spline representation for cubic splines over
  {P}owell-{S}abin triangulations.
\newblock Comput. Aided Geom. Design \textbf{37}, 42--56 (2015)

\bibitem{GeoPDEs}
V\'azquez, R.: A new design for the implementation of isogeometric analysis in
  {O}ctave and {M}atlab: {G}eo{PDE}s 3.0.
\newblock Comput. Math. Appl. \textbf{72}(3), 523--554 (2016)

\bibitem{Beirao2011}
Beir\~ao~da Veiga, L., Buffa, A., Rivas, J., Sangalli, G.: Some estimates for
  {$h$}-{$p$}-{$k$}-refinement in isogeometric analysis.
\newblock Numer. Math. \textbf{118}(2), 271--305 (2011)

\bibitem{Beirao2014}
Beir\~ao~da Veiga, L., Buffa, A., Sangalli, G., V\'azquez, R.: Mathematical
  analysis of variational isogeometric methods.
\newblock Acta Numer. \textbf{23}, 157--287 (2014)

\bibitem{Vuong2011}
Vuong, A.V., Giannelli, C., J\"uttler, B., Simeon, B.: A hierarchical approach
  to adaptive local refinement in isogeometric analysis.
\newblock Comput. Methods Appl. Mech. Engrg. \textbf{200}(49-52), 3554--3567
  (2011)

\bibitem{Zak2021}
\.Zak, A., Waszkowiak, W.: A spline-based {FE} approach to modelling of high
  frequency dynamics of 1-{D} structures.
\newblock Comput. Math. Appl. \textbf{104}, 14--33 (2021)

\end{thebibliography}
	
\end{document}